\documentclass[a4paper, 11pt]{article}
\usepackage{hyperref}
\usepackage[T1]{fontenc}
\usepackage{amsfonts}
\usepackage[english]{babel}
\usepackage{a4wide,times}

\usepackage[latin1]{inputenc}
\usepackage{amssymb}
\usepackage{graphicx}
\usepackage{epsfig}
\usepackage{amsbsy}
\usepackage{verbatim}
\usepackage{color}
\usepackage{mathrsfs}
\usepackage{makeidx}
\usepackage{amsmath}
\usepackage{amsthm}

\theoremstyle{plain}
\newtheorem{thm}{Theorem}[section]
\newtheorem{cor}[thm]{Corollary}
\newtheorem{lem}[thm]{Lemma}
\newtheorem{prop}[thm]{Proposition}

\theoremstyle{definition}
\newtheorem{defn}{Definition}[section]
\usepackage{dsfont}	

\theoremstyle{remark}
\newtheorem{rem}{\bf Remark}[section]
\theoremstyle{remark}
\newtheorem{example}{\bf Example}
\newtheorem{com*}{\bf Comment}

\usepackage{fancyhdr}

\makeatletter
\def \newequation#1#2{
   \@definecounter{#1}
   \@namedef{the#1}{\hbox{#2}}
   \@namedef{#1}{$$\refstepcounter{#1}}
   \@namedef{end#1}{
      \eqno \csname the#1\endcsname $$\global\@ignoretrue
      }
}
\makeatother
\newequation{E1}{($E_{b,\sigma,W}$)}
   
\makeatletter
\def \newequation#1#2{
   \@definecounter{#1}
   \@namedef{the#1}{\hbox{#2}}
   \@namedef{#1}{$$\refstepcounter{#1}}
   \@namedef{end#1}{
      \eqno \csname the#1\endcsname $$\global\@ignoretrue
      }
   }
\makeatother
\newequation{hyp3}{($\mathcal{H}_{b,\sigma}$)}

\makeatletter
\def \newequation#1#2{
   \@definecounter{#1}
   \@namedef{the#1}{\hbox{#2}}
   \@namedef{#1}{$$\refstepcounter{#1}}
   \@namedef{end#1}{
      \eqno \csname the#1\endcsname $$\global\@ignoretrue
      }
   }
\makeatother
\newequation{shleg}{(ShLeg)}

\makeatletter
\def \newequation#1#2{
   \@definecounter{#1}
   \@namedef{the#1}{\hbox{#2}}
   \@namedef{#1}{$$\refstepcounter{#1}}
   \@namedef{end#1}{
      \eqno \csname the#1\endcsname $$\global\@ignoretrue
      }
   }
\makeatother
\newequation{kl}{(KL)}

\makeatletter
\def \newequation#1#2{
   \@definecounter{#1}
   \@namedef{the#1}{\hbox{#2}}
   \@namedef{#1}{$$\refstepcounter{#1}}
   \@namedef{end#1}{
      \eqno \csname the#1\endcsname $$\global\@ignoretrue
      }
   }
\makeatother
\newequation{haar}{(Haar)}

\title{Conditional hitting time estimation in  a nonlinear filtering model by the Brownian bridge method}
\author{ 
{\sc  Christophe Profeta} \thanks{F\'ed\'eration de Math\'ematiques d'Evry,  Laboratoire d'Analyse et de Probabilit\'es, 23 Boulevard de France, 91037 Evry,  
e-mail: {\tt christophe.profeta@univ-evry.fr} } \ \ \ 
{\sc  Abass Sagna} \thanks{F\'ed\'eration de Math\'ematiques d'Evry,  Laboratoire d'Analyse et de Probabilit\'es, 23 Boulevard de France, 91037 Evry,  \& ENSIIE, e-mail: {\tt abass.sagna@ensiie.fr}. Both author's researches  are supported by an AMaMeF exchange grant and the ``Chaire Risque de Cr\'{e}dit'' of the French Banking Federation.  } 
}

\date{}

\begin{document}

\maketitle

\begin{abstract}
The model consists of  a  signal process  $X$ which  is a general  Brownian diffusion process and an observation process $Y$, also a diffusion process,  which is supposed to be correlated to the signal process. We suppose that  the process $Y$ is observed  from time $0$ to $s>0$  at discrete times and aim to estimate, conditionally on these observations, the probability that the non-observed process $X$ crosses a  fixed barrier after a given time $t>s$. We formulate  this problem as  a usual  nonlinear filtering problem  and use optimal quantization and Monte Carlo simulations techniques to estimate the involved quantities.   
\end{abstract}

\section{Introduction}

We consider in this  work a nonlinear filtering model where the signal process $X$ and the observation process $Y$  evolve following the stochastic differential equations:

\begin{equation} \label{EqSignalProcess}
 \begin{cases}
 dX_t = b(X_t,t) dt + \sigma(X_t,t) dW_t,  &   X_0=x_0, \\
 dY_t = h(Y_t,X_t,t) dt + \nu(Y_t,t) dW_t +  \delta(Y_t,t) d\widetilde{W}_t, &  Y_0=y_0.
 \end{cases} 
 \end{equation}

\noindent
In these equations, $W$ and $\widetilde W$ are two independent  standard real valued Brownian motions.  We suppose that the functions $b$, $\sigma$, $h$, $\nu$, and $\delta$ are Lipschitz and that, for every $(x,t) \in  ]0,+\infty)^2$,  $\delta (x,t) > 0$, $\nu(x,t)>0$ and  $\sigma(x,t)>0$.    

\noindent
 Let  $ {\bf a}$ be a real number such that $0 <  {\bf a} < x_0$  and   let 
$$ \tau_{\bf a}^X = \inf\{ u \geq 0,  X_u \leq {\bf a}  \} $$
be  the first hitting time of the barrier ${\bf a}$ by the signal process.  As  usually, we consider that  $\inf \emptyset = + \infty$.    Our aim is to estimate the  distribution of the conditional hitting time

\begin{equation} \label{EqDisCondHitTime}
\mathds P \left(\tau_{\bf a}^X> t \vert\, \mathcal F_s^Y\right)
\end{equation} 
for $t \geq s>0$  and where  $(\mathcal F_t^Y)_{t \geq 0}$ is the filtration generated by the observation process $Y$: 
$$\mathcal F_s^Y = \sigma(Y_u, u \leq s).$$   More generally, we shall denote by $(\mathcal F_t^Z)_{t\geq0}$ the filtration generated by the process $Z$.
  
Such a problem arises for example in credit risk  when modeling a credit event in a structural model as  the first hitting time of a barrier ${\bf a}$ by  the firm value process $X$. Investors are supposed to have no access to the  true value of  the firm but only to the observation process  $Y$, which is  correlated  to  the value of the firm   (see e.g.  \cite{CocGemJea, DufLan}). We will typically suppose that we observe the process $Y$  at regular discrete times $t_0 = 0 <t_1< \ldots < t_m=s$  over the time interval $[0,s]$ and  intend for estimating the quantity
$$ \mathds P \left(\tau_{\bf a}^X> t \vert\, Y_{t_0},\ldots,Y_{t_m}\right)  $$
  for every $t \geq s$.   Note that, if $t < s$, then, applying the Markov property to the diffusion $Y$, the computations boil down to the case $s=t$.   In \cite{CalSag}, the quantity 
$$ \mathds P\left(\inf_{u \in [s,t]} X_u>{\bf a}\vert  Y_{t_0},\ldots,Y_{t_m}  \right) $$
has been estimated by a hybrid Monte Carlo-Optimal quantization method in the  case where the observation process dynamics is given by: 
 $$dY_t = Y_t\left(h (X_t) dt + \nu(t) dW_t +  \delta(t) d\widetilde{W}_t \right), \quad   Y_0=y_0,$$
where $\nu$ and $\delta$ are deterministic functions. However,  the approach used in the previous work   does  not apply to our framework because  we want to compute the conditional distribution of a function of the whole trajectory of the signal process  from $0$ to $t$ given the observations from $0$ to $s$, with $s \leq t$.\\

\noindent
\begin{example}\label{exa:BS}
 A particular case of model (\ref{EqSignalProcess}) one may consider is the following ``Black-Scholes'' case:
 
 \begin{equation}
\label{DinBS}
\begin{cases}
dX_t  =  X_t ( \mu dt  +  \sigma  dW_t), &  X_0 = x_0, \\
d Y_t  =  Y_t (r dt + \nu d W_t + \delta d \widetilde  W_t), & Y_0=y_0,
\end{cases}
\end{equation}
so that
\begin{equation}  \label{EqReturnYParticCase}
\frac{dY_t}{Y_t} =  \frac{\nu}{\sigma}\frac{dX_t}{X_t} + \big(r -\mu \frac{\nu}{\sigma} \big) dt + \delta d \widetilde{W}_t,
\end{equation}
or
\begin{equation}  \label{EqReturnYPartiCase1}
Y_t = \frac{y_0}{(x_0)^{\nu/\sigma}}  X_t^{\nu/\sigma} \exp\left(\delta \widetilde W_t + \left(r-\frac{\nu^2+\delta^2}{2}-\frac{\mu\nu}{\sigma} + \frac{\nu \sigma}{2} \right)t \right).
\end{equation}
Observe that  setting   $r = \mu$ and $\sigma=\nu$ yields 

$$  \frac{dY_t}{Y_t} =  \frac{dX_t}{X_t}  + \delta d \widetilde{W}_t, $$
 meaning that the return on $Y$ is  the return on $X$ affected by a noise (see  e.g.  \cite{CocGemJea}). \\
Of course, in this case, we may compute theoretically the expression  (\ref{EqDisCondHitTime}) by noticing that: 
\begin{align*}
 \tau_{\bf a}^X &= \inf\{ u \geq 0,  X_u \leq {\bf a}  \}\\
 &=  \inf\left\{ u \geq 0,\;  \delta \widetilde W_u -\frac{\delta^2}{2}u \geq    \ln(Y_u) + \ln\left(\frac{x_0}{y_0} \right) - \ln(\bf a) \right\}
  \end{align*}
and that, conditionally to $\mathcal F_s^Y$, the process $\left(\widetilde{W}_u, \; u\leq s\right)$ has the same law as 
 $$\left(  \frac{\sigma}{\sqrt{\sigma^2+\delta^2}}B_u+ \frac{\delta}{\sigma^2+\delta^2}\left(\ln(Y_u) +\ln\left(\frac{x_0}{y_0}\right)+\frac{\sigma^2+\delta^2}{2}u\right),\; u\leq s\right)$$
 where $B$ is a standard Brownian motion independent from the process $Y$.
 This follows from the fact that, since $W$ and $\widetilde W$ are two independent Brownian motion, so are
$$B=\frac{\sigma W + \delta \widetilde W}{\sqrt{\sigma^2 + \delta^2}}\quad \text{ and }\quad \widetilde B=\frac{\delta W - \sigma \widetilde W}{\sqrt{\sigma^2 + \delta^2}},$$
and from the relations :
$$B_t= \frac{1}{\sqrt{\sigma^2 + \delta^2}} \left(\ln(Y_t) +\ln\left(\frac{x_0}{y_0}\right)+\frac{\sigma^2+\delta^2}{2}t\right)\quad \text{ and }\quad \widetilde{W}_t = \frac{\delta B_t - \sigma \widetilde B_t}{\sqrt{\sigma^2 + \delta^2}}.$$
 Therefore, in this particular setting, the problem boils down to the  computation of the first passage time of a Brownian motion to a curved boundary. We refer to \cite{CocGemJea} for some similar, and far more general, considerations. 
\end{example}

\begin{example}
Another particularly simple case is given when both the signal and the observation process evolve following the Ornstein-Uhlenbeck dynamics:
\begin{equation}
\begin{cases}
dX_t  =  \lambda (\theta- X_t)  dt  +  \sigma  dW_t, &  X_0 = x_0, \\
d Y_t  =  \lambda ( \theta - Y_t)  dt + \sigma d W_t + \delta d \widetilde  W_t, & Y_0=y_0.
\end{cases}
\end{equation}
In this case, we have :
$$X_t= Y_t -(y_0- x_0) e^{-\lambda t} - \delta e^{-\lambda t} \int_0^t e^{\lambda u} d\widetilde W_u,$$ 
so that
$$\tau_{\bf a}^X = \inf\left\{ u \geq 0, \; Y_u- {\bf a}-(y_0- x_0) e^{-\lambda u}  \leq \delta e^{-\lambda u} \int_0^u e^{\lambda v} d\widetilde W_v  \right\}$$
and, conditionally to $\mathcal F_s^Y$, the process $\displaystyle \left( e^{-\lambda u} \int_0^u e^{\lambda v} d\widetilde W_v, \; u\leq s  \right)$ has the same law as 
$$\left(\frac{\sigma}{\sqrt{\sigma^2+\delta^2}} U_u + \frac{\delta e^{\lambda u}}{\sigma^2+\delta^2} \left( Y_u - y_0e^{-\lambda u}- \theta(1-e^{-\lambda u}) \right),\; u\leq s  \right)$$
where $U$ is an Ornstein-Uhlenbeck process with parameters $\lambda$ and 0 started from 0, and independent from $Y$.
This follows from Knight's representation theorem combined with the same ideas as above.\\
 \end{example}

The rest of the paper is organized as follows: in Section \ref{sec2}, we state and prove the main theorems, i.e. we give an approximation of the expectation (\ref{EqDisCondHitTime}) when $X$ is replaced by its continuous Euler scheme $\bar X$, see especially Subsection \ref{sec2.3}. Then, in Section 3, we introduce some numerical tools in order to compute the quantities involved. We finally conclude the paper by a few simulations.

\section{Estimation of the conditional survival probability}	\label{sec2}
	
	\subsection{Preliminaries  results}

To deal with the computation of the conditional hitting time (\ref{EqDisCondHitTime}), we introduce  the first hitting time from time $s$:
$$ \tau^{(s)} =\inf\{ u \geq s,   X_u \leq {\bf a}  \}.$$
Define furthermore the filtration  $(\mathcal G_t)_{t \geq 0}$  by, for every $t \geq 0$,
\begin{equation}\label{EqDefSigAlgG}
 \mathcal G_t = \sigma\{(W_u,\widetilde{W}_u), u \leq t  \}.
 \end{equation}

\noindent
We have the following result.

\begin{lem} \label{LemPreliminaries}
For every $t \geq s$,
\begin{eqnarray}  \label{EqLemPreliminaries}
\mathds P \big(\tau_{\bf a}^X > t \vert\, \mathcal F_s^Y\big)  &  = &  \mathds E \Big[  {1 \! \! 1}_{\{ \tau_{\bf a}^{X} >s \}}  F(s,t,X_s) \vert\, \mathcal F_s^Y\Big]
\end{eqnarray}
with 
$$   F(s,t,X_s)  = \mathds P\left( \inf_{u \in [s,t]} X_u >{\bf a } \vert X_s \right). $$
\end{lem}

\begin{proof}[{\bf Proof}]
If $t\geq s$, one has
$$  \{ \tau_{\bf a}^X>t \} = \{\tau_{\bf a}^X > s   \} \cap \{ \tau^{(s)} >t\},$$
so that 
\begin{eqnarray*}
\mathds P \big(\tau_{\bf a}^X> t \vert\, \mathcal F_s^Y\big)  & =  &   \mathds P \big( \{\tau_{\bf a}^X > s   \} \cap \{ \tau^{(s)} >t\} \vert\, \mathcal F_s^Y\big) \\
&= &  \mathds E \left[ \mathds E\left[ {1 \! \! 1}_{\{ \tau_{\bf a}^X>s \}} {1 \! \! 1}_{\{ \tau^{(s)}>t \}}   \vert \mathcal G_s \right] \vert \mathcal{F}_s^Y  \right],
\end{eqnarray*}
where the $\sigma$-algebra $\mathcal G_s$ is defined in (\ref{EqDefSigAlgG}).
Hence, using the fact that $\widetilde{W}$ is independent from $W$  and the Markov property of the process $X$:  
\begin{eqnarray*}
\mathds P \big(\tau_{\bf a}^X> t \vert\, \mathcal F_s^Y\big)  & =  &   \mathds E \left[ {1 \! \! 1}_{\{ \tau_{\bf a}^X>s \}} \mathds E\left[  {1 \! \! 1}_{\{ \tau^{(s)}>t \}}   \vert \mathcal G_s \right] \vert \mathcal{F}_s^Y  \right] \\
& = &  \mathds E \left[ {1 \! \! 1}_{\{ \tau_{\bf a}^X>s \}} \mathds E\left[  {1 \! \! 1}_{\{ \tau^{(s)}>t \}}   \vert \mathcal F_s^W \right] \vert \mathcal{F}_s^Y  \right] \\
 & = & \mathds E \left[ {1 \! \! 1}_{\{ \tau_{\bf a}^X>s \}} \mathds P \big( \tau^{(s)}>t     \vert \mathcal F_s^X \big) \vert \mathcal{F}_s^Y  \right] \\
 & = & \mathds E \Big[ {1 \! \! 1}_{\{ \tau_{\bf a}^X>s \}} \mathds P \big( \inf_{u \in [s,t]} X_u >{\bf a}     \vert  X_s \big) \vert \mathcal{F}_s^Y  \Big].
\end{eqnarray*}
This shows assertion (\ref{EqLemPreliminaries}). 
\end{proof}

\noindent
Note that in general, there is no closed-form expression for computing $\displaystyle \mathds P \big( \inf_{u \in [s,t]} X_u >{\bf a} \vert  X_s \big)$, except in some very special cases, such as Brownian motion or Bessel processes... We shall therefore need an approximation of this expression, which is the purpose of the next subsection. 

\subsection{The continuous Euler scheme}

Let us  denote by $\bar X$ the continuous Euler scheme process  associated  to the signal process  $X$ and defined by 
$$
\bar{X}_s  = \bar{X}_{\underline{s}} + b(\bar{X}_{\underline{s}}, \underline{s}) (s - \underline{s}) + \sigma(\bar{X}_{\underline{s}}, \underline{s}) (W_{s} - W_{\underline{s}}), \quad  \bar{X}_0 = x_0,
$$
with $\underline{s} = t_k $ if $s \in [t_k,t_{k+1})$.
To estimate the distribution of the  conditional hitting time given in (\ref{EqDisCondHitTime}), consider  that we observe  the process $Y$ at discrete and regularly spaced  times: $t_0, t_1,\dots, t_m$, with   $0=t_0< \dots<t_m=s < t_{m+1}<\dots< t_n=t$.  In our general model (see Equation (\ref{EqSignalProcess})), the discrete time observation processes $\{\bar X_{t_k},\ k=0,\dots,n\}$ and  $\{Y_{t_k}, \  k=0,\dots,m\}$ are  obtained from Euler scheme as:
\setlength\arraycolsep{1pt}
\begin{equation}  
\begin{cases} 
 \bar X_{t_{k+1}} = \bar{X}_{t_k} +  b(\bar{X}_{t_k}, t_k) \Delta_k + \sigma(\bar X_{t_k},t_k) (W_{t_{k+1}}-W_{t_k})   \\
Y_{t_{k+1}} = Y_{t_k}  + h(Y_{t_k}, \bar X_{t_k}, t_k) \Delta_k + \nu(Y_{t_k},t_k) (W_{t_{k+1}}-W_{t_k}) + \delta(Y_{t_k},t_k) (\widetilde W_{t_{k+1}} -\widetilde W_{t_{k}})
 \end{cases}
 \end{equation} 
 \setlength\arraycolsep{2pt}
where $k=0,\dots,m-1$ for the observation process, $k$ going till $n-1$ for the signal process and $\Delta_k = t_{k+1}-t_{k}$. Note that if $t>s$, the number of discretization steps over $[0,s]$ may differ from the number of discretization steps over $[s,t]$ so that we choose  
\begin{equation*} 
  \left \{ \begin{array}{ll}
 t_k = \frac{ks}{m} & \textrm{ for }  k=0,\dots,m\\
 \\
 t_k = s+\frac{k(t-s)}{n}  &   \textrm{ for } k=m+1,\dots, n.
\end{array}  \right.
\end{equation*}

\noindent
 Supposing that  we have observed the trajectory $(Y_{t_0},\ldots,Y_{t_m})$ of the observation process  $Y$  we aim  to estimate 
$$ \mathds P(\tau_{\bf a}^X>t_n  \vert\,  Y_{t_0},\dots, Y_{t_m} ) $$ 
 by
\begin{equation} \label{EqDisCondHitTimeEstim}
\mathds P \big(\tau_{\bf a}^{\bar X} > t_n \vert\,  Y_{t_0},\dots, Y_{t_m}\big),
\end{equation} 
where 
$$ \tau_{\bf a}^{\bar X} = \inf\{ u \geq 0,  \bar{X}_u \leq {\bf a}  \}.$$
 To this end, we will use a useful result called   the regular Brownian bridge method. This result  recalled  below  allows us to compute  the distribution of the minimum (or the maximum) of the continuous Euler scheme $\bar{X}$ of the process $X$ over the time interval $[0,t]$, given its values at discrete and regular time observation points $0=t_0 < t_1 < \dots < t_n=t$ (see, e.g.  \cite{Gla, GobThese}).

\begin{lem} \label{LemBrowBridgeMethod} Let $(X_t)_{t \geq 0}$ be a diffusion process with dynamics given by 
$$ X_t = x  + \int_0^t b(X_u,u) du + \int_0^t \sigma(X_u,u) dW_u  $$
and let $(\bar{X}_t)_{t \geq 0}$ be its associated continuous Euler process. Then, the following equality in law holds:
\begin{equation} 
\mathcal{L} \Big(\min_{u \in [0,t]} \bar{X}_u \vert \bar{X}_{t_k} = x_k, k=0,\cdots,n \Big)   = \mathcal{L} \Big(\min_{k=0, \cdots,n-1} \left(H_{\Delta_k\sigma^2(x_k,t_k)}^{x_{k},x_{k+1}}\right)^{-1} (\Lambda_k) \Big),
\end{equation}
where $(\Lambda_k)_{k=0, \cdots,n-1}$ are $i.i.d.$ random variables uniformly distributed over the unit interval and $ \left(H_{\Delta_k\sigma^2(x_k,t_k)}^{x_{k},x_{k+1}}\right)^{-1} $  is the inverse function of the conditional cumulative function  $ H_{\Delta_k\sigma^2(x_k,t_k)}^{x_{k},x_{k+1}} $, defined by

\begin{equation} \label{DefCumulDistFunc}
H_{\Delta_k\sigma^2(x_k,t_k)}^{x_{k},x_{k+1}} (u) : =  \left \{ \begin{array}{ll}
\exp \Big( - \frac{2}{ \Delta_k  \sigma^2(x_k, t_k)}(u-x_k) (u-x_{k+1}) \Big) &   \textrm{if }  u \leq \min(x_k,x_{k+1})  \\
 1  &   \textrm{otherwise.}
\end{array}  \right.
\end{equation}
\end{lem}
In the following, we shall replace the expression $\sigma(x_k,t_k)$ in the expression of $H$ by $\sigma_k$ and rather write: $ H_{\Delta_k\sigma_k^2}^{x_{k},x_{k+1}} $.

\begin{proof}[{\bf Short proof}]
Observe first that, conditionally to $\{\bar{X}_{t_k} = x_k;\; k=0,\cdots,n \}$,
the random variables  $\displaystyle \left\{\min_{u \in [t_k,t_{k+1}]} \bar{X}_u ;\; k=0,\ldots, n-1\right\}$ are mutually independent, thanks to the independent increments property of Brownian motion. Then, it suffices to notice that computing the law of $\displaystyle\min_{u \in [t_k,t_{k+1}]} \bar{X}_u $ conditionally to $\{ \bar X_{t_k}=x_k,\bar X_{t_{k+1}}=x_{k+1}\}$ amounts to computing the law of the hitting times of the bridge of a Brownian motion with drift. This law is well-known to be independent from the drift (i.e. from the function $b$ here) and to be given by an expression such as (\ref{DefCumulDistFunc}).
\end{proof}

 \noindent
 In the rest of the paper,  the function  $1-  H_{\Delta_k\sigma_k^2}^{x_{k},x_{k+1}}({\bf a}) $  will be often used. We shall denote it by   $G_{\Delta_k\sigma_k^2}^{x_{k},x_{k+1}}({\bf a})$, so that
 \begin{equation} \label{EqDefOfG}
   G_{\Delta_k\sigma_k^2}^{x_{k},x_{k+1}}({\bf a})= \left( 1-\exp\left(-  \frac{ 2 (x_k-{\bf a}) (x_{k+1}-{\bf a})}{ \Delta_k \sigma^2  (x_k, t_k) }\right) \right) {1 \! \! 1}_{\{x_k \geq {\bf a};\; x_{k+1} \geq  {\bf a}\} }.  
   \end{equation}
Now, in our set-up, we shall apply the Brownian bridge method to obtain the following lemma,  which is taken from \cite{CalSag}:

 \begin{lem} \label{LemProveThmMainResultPhi}
 We have 
 \begin{equation}
 \mathds P\left(\inf_{u \in [t_m,t_n]} \bar{X}_u>{\bf  a} \big \vert (\bar X_{t_k})_{k=0,\dots,m} \right) =   \prod_{k=m}^{n-1}  G_{\Delta_k\sigma_k^2}^{\bar{X}_{t_k},\bar{X}_{t_{k+1}}}({\bf a}). 
 \end{equation}
 \end{lem}

\begin{proof}[{\bf Proof.}]  
It follows from Lemma \ref{LemBrowBridgeMethod}  that 
\begin{eqnarray*}
 \mathds P\left(\inf_{u \in [t_m,t_n]} \bar{X}_u> a  \big\vert (\bar X_{t_k})_{k=0,\dots,m} \right)  & = &\mathbb P \left(  \min_{k=m,\cdots,n-1}\left(H_{\Delta_k\sigma_k^2}^{\bar{X}_{t_k},\bar{X}_{t_{k+1}}}\right)^{-1} (\Lambda_k)  > {\bf a}  \right) . \nonumber
\end{eqnarray*}
Since the function $H_{\Delta_k\sigma_k^2}^{z_{k},z_{k+1}}(\cdot) $ is non-decreasing and the $\Lambda_{k}$'s are i.i.d. uniformly distributed random variables, this reduces to:
\begin{eqnarray*}
\mathbb P \left( \min_{ u \in [t_m,t_n]} \bar X_u>a \big \vert (\bar X_{t_k})_{k=0,\dots,m} \right) & = &
 \prod_{k=m}^{n-1} \mathbb P \left(  \Lambda_{k}   \geq  H_{\Delta_k\sigma_k^2}^{\bar{X}_{t_k},\bar{X}_{t_{k+1}}}({\bf a})   \right)  \\
& = &  \prod_{k=m}^{n-1} \left(1-H_{\Delta_k\sigma_k^2}^{\bar{X}_{t_k},\bar{X}_{t_{k+1}}}({\bf a})  \right).
\end{eqnarray*}
This  completes the proof.
\end{proof}

\subsection{Main theorem}\label{sec2.3}

We now state and prove the main theorems of this section. We first show that the conditional hitting time of the continuous Euler process $\bar X$ given the discrete path observations $Y_{t_0}, \dots, Y_{t_m}$  may be written as an expectation of an explicit functional of the discrete path $\bar X_{t_0}, \dots, \bar X_{t_m}$ of the signal given the observations  $Y_{t_0}, \dots, Y_{t_m}$. Therefore, we reduce the initial problem to the characterization of the conditional distribution of  $\bar X_{t_0}, \dots, \bar X_{t_m}$  given   $Y_{t_0}, \dots, Y_{t_m}$.  \\

\begin{thm}  \label{ThmMainResultPhi}
Set  $\bar{X}^m= (\bar{X}_{t_0},\dots,\bar{X}_{t_m})$ and $Y^m=(Y_{t_0},\dots,Y_{t_m})$.   We have:
\begin{equation}   \label{EqForFixedObservation}
\mathds P\big(  \tau_{\bf a}^{\bar X} > t_n \vert  \,  Y_{t_0},\dots, Y_{t_m} \big)  = \Psi(Y_{t_0},\dots,Y_{t_m}),
\end{equation}
where for $y^m:=(y_0,\dots,y_{m})$, 
\begin{equation} \label{EqForFixedObservationUm}
\Psi (y^m) = \mathds E\big[ \Phi \big(\bar{X}^m\big) \big \vert Y^m=y^m\big].
\end{equation} 
The function $\Phi$ is defined for every   $z^{m}=(z_0,\dots,z_m)$ by
  $$ \Phi(z^m) = \bar{F}(t_m,t_n,z_m)\prod_{k=0}^{m-1}  G_{\Delta_k\sigma_k^2}^{z_{k},z_{k+1}}({\bf a}),$$
with 
$$ \bar{F}(t_m,t_n,x)  =  \mathds E \left[ \prod_{k=m}^{n-1}   G_{\Delta_k\sigma_k^2}^{\bar{X}_{t_k},\bar{X}_{t_{k+1}}}({\bf a}) \big\vert\, \bar{X}_{t_m}=x  \right]. $$

\end{thm}

\begin{proof}[{\bf Proof}]
 We may show  similarly to  Equation (\ref{EqLemPreliminaries})  in Lemma \ref{LemPreliminaries} that 
$$  \mathds P \big(\tau_{\bf a}^{\bar X} > t_n  \big\vert\, Y_{t_0},\dots,Y_{t_m}\big)  = \mathds E \left[  {1 \! \! 1}_{\{ \tau_{{\bf a}}^{\bar X} >t_m \}}  \bar{F}(t_m,t_n,\bar{X}_{t_m})\vert\, Y_{t_0},\dots,Y_{t_m} \right], $$
with  
$$\bar{F}(t_m,t_n,\bar{X}_{t_m}) = \mathds P\left(\inf_{u \in [t_m,t_n]} \bar{X}_u> {\bf a}\big \vert \bar{X}_{t_m} \right).$$
It follows from Lemma \ref{LemProveThmMainResultPhi}  that
$$ \bar{F}(t_m,t_n,\bar{X}_{t_m})  =  \mathds E\left[ \mathbb P \left( \min_{ u \in [t_m,t_n]} \bar X_u>{\bf a} \big \vert (\bar X_{t_k})_{k=0,\dots,m} \right)    \big \vert \bar{X}_{t_m}\right]   =\mathds E \left[ \prod_{k=m}^{n-1} G_{\Delta_k\sigma_k^2}^{\bar{X}_{t_k},\bar{X}_{t_{k+1}}}({\bf a}) \big\vert\, \bar{X}_{t_m}  \right].$$
On the other hand, combining  a  successive conditioning rule with the independence between $W$ and $\widetilde W$  gives :
\begin{eqnarray*}
  \mathds P \big(\tau_{\bf a}^{\bar{X}} > t_n  \vert\, Y_{t_0},\dots,Y_{t_m}\big)  & = &  \mathds E \left[ \mathds E \left[  {1 \! \! 1}_{\{ \tau_{\bf a}^{\bar X} >t_m \}}  \bar{F}(t_m,t_n,\bar{X}_{t_m})\vert   (W_{t_k},\widetilde W_{t_k})_{k=0,\dots,m} \right] \vert\, Y^m \right] \\
  & = &  \mathds E \left[ \mathds E \left[  {1 \! \! 1}_{\{ \tau_{\bf a}^{\bar X} >t_m \}}  \bar{F}(t_n,t_m,\bar{X}_{t_m}) \vert (W_{t_k})_{k=0,\dots,m} \right] \vert\, Y^m \right] \\
  & = &  \mathds E \left[  \bar{F}(t_m,t_n,\bar{X}_{t_m})  \mathds P\left( \tau_{ {\bf a}}^{\bar X} >t_m  \vert (W_{t_k})_{k=0,\dots,m} \right) \vert\, Y^m \right] \\
   & = &  \mathds E \left[  \bar{F}(t_m,t_n,\bar{X}_{t_m})  \mathds P\left( \tau_{ {\bf a}}^{\bar X} >t_m  \vert (\bar X_{t_k})_{k=0,\dots,m} \right) \vert\, Y^m \right].
\end{eqnarray*}
 Now,  using   Lemma \ref{LemProveThmMainResultPhi} once again, we have
  $$  \mathds P\left( \tau_{ {\bf a}}^{\bar X} >t_m   \vert (\bar X_{t_k})_{k=0,\dots,m} \right)  = \prod_{k=0}^{m-1}  G_{\Delta_k\sigma_k^2}^{\bar{X}_{t_k},\bar{X}_{t_{k+1}}}({\bf a}) $$
  so that 
  $$  \mathds P \big(\tau_{\bf a}^{\bar X} > t_n  \vert\, Y_{t_0},\dots,Y_{t_m}\big) = \mathds E\left[ \Phi(\bar X^m) \,\vert \, Y_{t_0},\dots,Y_{t_m} \right] $$
  where  for every $x^m:=(x_0,\dots,x_m)$,
$$ \Phi(x^m) = \bar{F}(t_m,t_n,x_m)\prod_{k=0}^{m-1}  G_{\Delta_k\sigma_k^2}^{x_{k},x_{k+1}}({\bf a}) .$$
This completes the proof.
\end{proof}

It remains now to characterize the law of $\bar{X}^m \vert Y^m=y^m$ appearing in (\ref{EqForFixedObservationUm}). This is the purpose of the following theorem, in which we shall write the conditional survival  probability  in a usual  form with respect to  standard nonlinear filtering problems. \\
\noindent
We set from now on $\sigma_k = \sigma(x_k, t_k),  b_k = b(x_k, t_k), \delta_k=\delta(y_k,t_k), h_k =h(y_k,x_k, t_k)$, and $\nu_k = \nu(y_k,t_k)$. We next give the main result of the paper.

\begin{thm}  \label{ThmMainResultPhiNLF}
We have:
\begin{equation}   \label{EqForFixedObservationNLF}
\mathds P\big(  \tau_{\bf a}^{\bar X} > t_n \vert  \,  Y_{t_0},\dots, Y_{t_m} \big)  = \Psi(Y_{t_0},\dots,Y_{t_m}),
\end{equation}
where for  $y^m:=(y_0,\dots,y_{m})$, 
\begin{equation}  \label{EqExplicitCompSurvivalProba}
\Psi (y^m) = \frac{\mathds E \big[\bar{F}(t_m,t_n,\bar X_{t_m}) K^{m} L_y^m \big] }{ \mathds E[L_y^m]},
\end{equation}
with 
$$   K^{m} =  \prod_{k=0}^{m-1}  G_{\Delta_k\sigma_k^2}^{\bar{X}_{t_k},\bar{X}_{t_{k+1}}}({\bf a})  \quad \textrm{and} \quad L_y^m =  \prod_{k=0}^{m-1} g_k(\bar X_{t_k},y_k;\bar X_{t_{k+1}},y_{k+1}).  $$
The function $g_k$  is defined by   
\begin{equation}  \label{EqDefFunctg_k}
 g_k(x_k,y_k;x_{k+1},y_{k+1}) =   \frac{1}{(2\pi \Delta_k)^{3/2} \sigma^2_k  \delta_k  }  \exp\Bigg(  -     \frac{\nu_k^2}{2 \delta^2_k \Delta_k} \Big( \frac{x_{k+1}  - m_k^1}{\sigma_k } -   \frac{y_{k+1} -   m_k^2}{\nu_k}  \Big)^2 \Bigg)  
\end{equation}

\noindent with  $m_k^1 := x_k + b_k \Delta_k\;$ and $\;m_k^2 := y_k + h_k  \Delta_k$.
\end{thm}

\begin{proof}[{\bf Proof}]
It follows from Theorem \ref{ThmMainResultPhi}  that
$$  \mathds P \big(\tau_{\bf a}^{\bar X} > t_n  \vert\, Y_{t_0},\dots,Y_{t_m}\big) = \mathds E\left[ \Phi(\bar X^m) \,\vert \, Y_{t_0},\dots,Y_{t_m} \right] $$
  where  for every $x^m:=(x_0,\dots,x_m)$,
$$ \Phi(x^m) = \bar{F}(t_m,t_n,x_m)\prod_{k=0}^{m-1}  G_{\Delta_k\sigma_k^2}^{x_{k},x_{k+1}}({\bf a}) .$$
It remains to characterize the conditional distribution of $\bar X^m$ given $Y^m$. Recall  that the dynamics of  the processes $(\bar X_{t_k})$ and $(Y_{t_k})$ are given  for $k =0,\dots,m$  by
%
  \setlength\arraycolsep{1pt}
\begin{equation}  
\begin{cases} 
 \displaystyle \bar X_{t_{k+1}} = \bar{X}_{t_k} +  b(\bar{X}_{t_k}, t_k) \Delta_k + \sigma(\bar X_{t_k},t_k) (W_{t_{k+1}}-W_{t_k})   \\
\displaystyle Y_{t_{k+1}} = Y_{t_k}  + h(Y_{t_k}, \bar X_{t_{k}},t_k) \Delta_k + \nu(Y_{t_k},t_k) (W_{t_{k+1}}-W_{t_k}) + \delta(Y_{t_k},t_k) (\widetilde W_{t_{k+1}} -\widetilde W_{t_{k}})
 \end{cases}  
 \end{equation} 
 \setlength\arraycolsep{2pt}

 \noindent 
Let $\mathds P_{x_{k},y_{k}}(\bar X_{t_{k+1}} \in dx_{k+1}, Y_{t_{k+1}} \in dy_{k+1})$ denote the density function of the couple $(\bar X_{t_{k+1}}, Y_{t_{k+1}})$ given $(\bar X_{t_k}, Y_{t_k}) = (x_k,y_k)$.   Using  Bayes formula and the Markov property of the process $(\bar X_{t_k},Y_{t_k})_k$ we get:
\begin{eqnarray*}
& & \mathds P_{x_0,y_0}  (  \bar X_{t_1} \in dx_1, \dots, \bar X_{t_m} \in dx_m \vert Y_{t_1} \in dy_1, \dots,Y_{t_m}  \in dy_m) \\
&   &  \qquad  \quad =   \frac{\mathds P_{x_0,y_0} \big((\bar X_{t_1} \in dx_1,Y_{t_1} \in dy_1); \dots; (\bar X_{t_m} \in dx_m,Y_{t_m} \in dy_m)   \big)}{ \mathds P_{x_0,y_0} (Y_{t_1} \in dy_1, \dots, Y_{t_m} \in dy_m)} \\
&   &  \qquad  \quad =    \frac{\mathds P_{x_0,y_0} \big(\bar X_{t_1} \in dx_1,Y_{t_1} \in dy_1) \times  \dots \times \mathds P_{x_{m-1},y_{m-1}} (\bar X_{t_m} \in dx_m,Y_{t_m} \in dy_m)}{ \mathds P_{x_0,y_0} (Y_{t_1} \in dy_1, \dots, Y_{t_m} \in dy_m)}. 
\end{eqnarray*}
For  every $x^m = (x_0,\dots,x_m)$ and $y^m=(y_0,\dots,y_m)$, set  :
$$  N_k(x^m; y^m) := \mathds P_{x_0,y_0} \big(\bar X_{t_1} \in dx_1,Y_{t_1} \in dy_1) \times  \dots \times \mathds P_{x_{m-1},y_{m-1}} (\bar X_{t_m} \in dx_m,Y_{t_m} \in dy_m). $$
With this notation, the denominator reads:
$$  \mathds P_{x_0,y_0} (Y_{t_1} \in dy_1, \dots, Y_{t_m} \in dy_m) = \int   N_k(x^m; y^m)  dx_1\ldots dx_m.$$

\noindent
On the other hand,  the random vector $ (\bar X_{t_{k+1}}, Y_{t_{k+1}}) \big  \vert (\bar X_{t_k},Y_{t_k}) = (x_k,y_k)$ has a Gaussian distribution with mean $m_{k}$ and covariance matrix $\Sigma_k$ given for every $k=0,\dots,m-1$  by  
 \begin{equation} \label{DefmkSigk} 
m_k =  \begin{pmatrix}
  x_k + b_k  \Delta_k   \\
y_k + h_k   \Delta_k\\ 
\end{pmatrix}, 
\qquad   \ 
\Sigma_k =  \Delta_k \begin{pmatrix}
\sigma^2_k  & \sigma_k \nu_k  \\ 
 \sigma_k \nu_k   & \quad  \nu^2_k + \delta^2_k
\end{pmatrix}.
\end{equation}
Then, for every $k=0,\dots,m-1$,  the density  $f_k(x_k,y_k;x_{k+1}, y_{k+1})$ of  $ (\bar X_{t_{k+1}}, Y_{t_{k+1}}) \big  \vert (\bar X_{t_k},Y_{t_k}) = (x_k,y_k)$  reads 
\begin{align}
f_k(x_k,y_k;x_{k+1},y_{k+1})   =   \frac{1}{2\pi \sigma_k  \delta_k  \Delta_k}  \exp\Bigg( & -  \frac{\nu^2_k +\delta^2_k}{2 \delta^2_k \Delta_k} \Big\{  \frac{ (x_{k+1}  - m_k^1)^2}{\sigma^2_k }   +    \frac{(y_{k+1} -   m_k^2)^2}{\nu^2_k + \delta^2_k}  \nonumber \\
&  -   \frac{2 \nu_k}{\sigma_k (\nu^2_k + \delta^2_k) } (x_{k+1}  - m_k^1)(y_{k+1} - m_k^2) \Big\} \Bigg) ,
\end{align}
where  $m_k^1 := x_k +b_k  \Delta_k\;$ and $\;m_k^2 := y_k + h_k  \Delta_k$.  As a consequence,
\begin{eqnarray}
\notag \mathds P \left(\tau_{\bf a}^{\bar X} > t_n  \vert\, (Y_{t_0},\dots,Y_{t_m}) =  y^m\right) & = &  \mathds E\left[ \Phi(\bar X^m) \,\vert \, (Y_{t_0},\dots,Y_{t_m}) =y^m \right] \\
\label{eqintphi} & = & \frac{\int \Phi(x^m)  N_k(x^m; y^m)  dx_1\ldots dx_m}{\int N_k(x^m; y^m)  dx_1\ldots dx_m}.
\end{eqnarray}

\noindent
Now, we know that the random variable  $\bar X_{t_{k+1}} \vert \bar{X}_{t_k} =x_{k}$ has a Gaussian distribution with mean $m_k^1$ and variance $\sigma^2_k  \Delta_k$.  Its  density  $P_k(x_k,x_{k+1})dx_{k+1}$ is therefore given by : 
$$  P_k(x_k,x_{k+1})dx_{k+1}  =  \frac{1}{\sqrt{2 \pi \Delta_k}\, \sigma_k } \exp \Big(-  \frac{(x_{k+1} - m_k^1)^2}{2 \sigma^2_k \Delta_k } \Big) dx_{k+1}.$$
Then, by definition of $N_k$, we may write:
\begin{align*}
&\int N_k(x^m; y^m)  dx_1\ldots dx_m\\
 =& \int \prod_{k=0}^{m-1}  f_k(x_k,y_k;x_{k+1}, y_{k+1}) dx_{k+1}\\
=& \int \prod_{k=0}^{m-1}  \frac{f_k(x_k,y_k;x_{k+1}, y_{k+1})}{P_k(x_k,x_{k+1})}  P_k(x_k,x_{k+1}) dx_{k+1}\\
=& \int \left(\prod_{k=0}^{m-1}  \frac{f_k(x_k,y_k;x_{k+1}, y_{k+1})}{P_k(x_k,x_{k+1})}\right)   \prod_{k=0}^{m-1}  P_k(x_k,x_{k+1}) dx_{k+1}\\
=& \int \left(\prod_{k=0}^{m-1} g_k(x_k,y_k;x_{k+1}, y_{k+1})\right) \mathds P_{x_0,y_0}  (\bar X_{t_1} \in dx_1, \dots, \bar X_{t_m} \in dx_m) =\mathds E[L_y^m],
\end{align*}
where $g_k$ is defined by (\ref{EqDefFunctg_k}) and the next-to-last equality follows from the Markov property of the process $(\bar X_{t_k}, \; k=0,\ldots,m)$.
Now, looking at the numerator of (\ref{eqintphi}), similar computations lead to:
\begin{align*}
&\int \Phi(x^m) N_k(x^m; y^m)  dx_1\ldots dx_m\\
 =&   \int \bar{F}(t_m,t_n,x_m)\prod_{k=0}^{m-1}  G_{\Delta_k\sigma_k^2}^{x_{k},x_{k+1}}({\bf a})  f_k(x_k,y_k;x_{k+1},y_{k+1})  dx_{k+1},\\
 =&  \int \bar{F}(t_m,t_n,x_m) \left(\prod_{k=0}^{m-1}  G_{\Delta_k\sigma_k^2}^{x_{k},x_{k+1}}({\bf a})\right)\\
&\qquad \qquad \times  \left( \prod_{k=0}^{m-1}  g_k(x_k,y_k;x_{k+1}, y_{k+1})\right)  \mathds P_{x_0,y_0}  (\bar X_{t_1} \in dx_1, \dots, \bar X_{t_m} \in dx_m)\\
=& \mathds E \big[\bar{F}(t_m,t_n,\bar X_{t_m}) K^{m} L_y^m \big]. 
\end{align*}
Finally, going back to  (\ref{eqintphi}), we obtain
$$ \mathds P \left(\tau_{\bf a}^{\bar X} > t_n  \vert\, (Y_{t_0},\dots,Y_{t_m}) =  y^m\right) =  \frac{\mathds E \big[\bar{F}(t_m,t_n,\bar X_{t_m}) K^{m} L_y^m \big] }{ \mathds E[L_y^m]}$$
which is the announced result.
\end{proof}

%
%
%
\noindent
\begin{rem}
Remark that more generally we have for every real valued bounded function $f$,
   \begin{equation}  \label{EqStandardFiltering}
   \mathds E[f(\bar X_{t_m}) \vert Y_{t_0}=y_0,\dots,Y_{t_m}= y_m] =  \frac{\mathds E \big[f(\bar X_{t_m})  L_y^m \big] }{ \mathds E[L_y^m]}.
   \end{equation}
   In this case we can estimate the right hand side quantity of (\ref{EqStandardFiltering}) using  recursive algorithms. In our setting, we can adapt these algorithms by noting  that the expression (\ref{EqExplicitCompSurvivalProba})  may be read in the similar form of (\ref{EqStandardFiltering}) as: 
         $$  \mathds P \left(\tau_{\bf a}^{\bar X} > t_n  \vert\, (Y_{t_0},\dots,Y_{t_m}) =  y^m\right) =  \frac{\mathds E \big[\bar{F}(t_m,t_n,\bar X_{t_m})  L_y^{'m} \big] }{ \mathds E[L_y^m]} $$
  where  
  $$   L_y^{'m} =  \prod_{k=0}^{m-1}   g_k(\bar X_{t_k},y_k;\bar X_{t_{k+1}},y_{k+1}) G_{\Delta_k\sigma_k^2}^{\bar{X}_{t_k},\bar{X}_{t_{k+1}}}({\bf a}) $$
  and where the involved functions are defined in Theorem \ref{ThmMainResultPhiNLF}.
\end{rem}

Considering the model given in Example \ref{exa:BS}, one may apply  the result of Theorem \ref{ThmMainResultPhiNLF} using the continuous Euler process $\bar X$ of the signal. However since in this framework the solutions of the stochastic differential equations are explicit, we refrain from using the Euler scheme to avoid adding  additional error.  In the next result we deduce a similar representation of the conditional survival probability  using the explicit solutions of the stochastic differential equations of the model (\ref{DinBS}).

\begin{cor} \label{CorExampleDinBS} Consider that the signal process $X$ and the observation process $Y$ evolve following the stochastic differential equations given in  (\ref{DinBS}):
 \begin{equation}
\begin{cases}
dX_t  =  X_t ( \mu dt  +  \sigma  dW_t), &  X_0 = x_0, \\
d Y_t  =  Y_t (r dt + \nu d W_t + \delta d \widetilde  W_t), & Y_0=y_0.
\end{cases}
\end{equation}
   Then, 
 \begin{equation}   \label{EqForFixedObservationBS}
\Psi (y^m) = \frac{\mathds E \big[F(t_m,t_n, X_{t_m}) K^{m} L_y^m \big] }{ \mathds E[L_y^m]},
\end{equation}
with 
$$   K^{m} =  \prod_{k=0}^{m-1}  G_{\Delta_k\sigma_k^2}^{X_{t_k},X_{t_{k+1}}}({\bf a})  \quad \textrm{and} \quad L_y^m =  \prod_{k=0}^{m-1} g_k(X_{t_k},y_k;X_{t_{k+1}},y_{k+1}).  $$
The function  $g_k$  is defined by   
\setlength\arraycolsep{1pt}
\begin{equation}
g_k(x_k,y_k;x_{k+1},y_{k+1}) =   \frac{1}{(2\pi \Delta_k)^{3/2}  \pi_k }  \exp\Bigg(  -  \frac{\nu^2}{2 \delta^2 \Delta_k} \Big( \frac{ \log x_{k+1}  - m_k^1}{\sigma } -   \frac{\log y_{k+1} -   m_k^2}{\nu}  \Big)^2 \Bigg) 
\end{equation}
with 
$$\pi_k :=  \sigma^2 \delta x_{k+1}^2 y_{k+1},\quad  m_k^1 := \log  x_k + (\mu - \frac{1}{2} \sigma^2 ) \Delta_k\quad \text{and}\quad  m_k^2 := \log y_k + (r - \frac{1}{2} \nu^2 - \frac{1}{2} \delta^2 )  \Delta_k.$$
\end{cor} 


\begin{proof}[{\bf Proof.}]  It follows from  It\^o formula that  for every $s \leq t$,
 \begin{equation}
\begin{cases}
X_t  =  X_s \exp\left( ( \mu - \frac{1}{2} \sigma^2) (t-s)  +  \sigma  (W_t -W_s) \right)  \\
Y_t  =  Y_s \exp \left(   (r - \frac{1}{2} \nu^2 - \frac{1}{2} \delta^2) (t-s) + \nu (W_t -W_s) + \delta (\widetilde  W_t-\widetilde  W_s) \right).
\end{cases}
\end{equation}
Then, for every $k=0,\dots,m-1$ the random vector $(X_{t_{k+1}},Y_{t_{k+1}}) \vert (X_{t_k}, Y_{t_k}) = (x_k,y_k)$ has a bivariate lognormal distribution with mean $m_k$ and covariance matrix $\Sigma$ given by
\begin{equation*}  
m_k =  \begin{pmatrix}
  \log x_k +  (\mu -\frac{1}{2} \sigma^2) \Delta_k   \\
\log y_k + (r-\frac{1}{2}\nu^2 - \frac{1}{2} \delta^2)  \Delta_k\\ 
\end{pmatrix}, 
\qquad   \ 
\Sigma =  \Delta_k \begin{pmatrix}
\sigma^2  & \sigma \nu   \\ 
\sigma \nu & \quad \nu^2+ \delta^2
\end{pmatrix}.
\end{equation*}
Hence, its density reads (setting $\mu_k=\sigma  \delta  x_{k+1} y_{k+1} $)
\setlength\arraycolsep{1pt}
\begin{eqnarray*}
f_k(x_k,y_k;x_{k+1},y_{k+1})   =   \frac{1}{2\pi  \Delta_k \mu_k}  \exp\Bigg( & - & \frac{\nu^2 +\delta^2}{2 \delta^2 \Delta_k} \Big\{  \frac{ (\log x_{k+1}  - m_k^1)^2}{\sigma^2 }   \\
&  - &  \frac{2 \nu}{\sigma (\nu^2 + \delta^2) } (\log x_{k+1}  - m_k^1)(\log y_{k+1} - m_k^2) \\
& + &  \frac{(\log y_{k+1} -   m_k^2)^2}{\nu^2 + \delta^2}   \Big\} \Bigg) ,
\end{eqnarray*}
where $m_k^1=\log x_{k} +(\mu - \frac{1}{2} \sigma^2) \Delta_k$ and  $m_k^2 := \log y_k + (r - \frac{1}{2} \nu^2 - \frac{1}{2} \delta^2 )  \Delta_k$.   Furthermore, for every $k=0,\dots,m-1$, the random variable $X_{t_{k+1}} \vert X_{t_k} = x_k$ has a lognormal distribution with mean $m_k^1$ and variance $\sigma^2$ so that its density distribution  $P_k(x_k,x_{k+1})$ is given by
$$P_k(x_k,x_{k+1}) =  \frac{1}{\sqrt{2 \pi  \Delta_k}\, \sigma x_{k+1} } \exp \Big(-  \frac{(\log x_{k+1} - m_k^1)^2}{2 \sigma^2 \Delta_k } \Big).$$
We then conclude the proof by using the same arguments than those of the proof of Theorem \ref{ThmMainResultPhiNLF}.
\end{proof}

There exist several methods to estimate the above representation of the conditional survival probability. These methods  involve, amount others,  Monte Carlo simulations and  optimal quantization methods.  Owing  to the numerical performance of the optimal quantization method due for example to its fast performability as soon as the optimal grids are obtained,  we will use optimal quantization methods to estimate the conditional survival probability. 

The use of  optimal quantization methods to estimate the filter supposes to have numerical access  to optimal (or stationary) quantizers of the marginals of the signal process.  One may use  the optimal vector quantization method  (as done in the seminal work \cite{PagPha} and used in \cite{CalSag}) to estimate these marginals.     This method  requires  the use of some algorithms, like stochastic algorithms or Lloyd's algorithm,  to obtain numerically the optimal (or stationary) quantizers of the marginals of the signal process.     Given  these optimal  quantizers,   the filter estimation is obtained  quite instantaneously. However,  the step of search of stationary quantizers is very time consuming.

      We propose  here an alternative method,  the (quadratic) marginal functional quantization method (introduce in \cite{Sag} to price barrier options),  to  quantize the marginals of the signal process.        The marginal  functional quantization method  consists first in considering the ordinary differential equation (ODE)  resulting to  the substitution of  the Brownian motion appearing in the dynamics of the  signal  process  by a quadratic quantization  of  the  Brownian motion. Then, by constructing some ``good'' marginal  quantization of the signal process based on the solution of the previous ODE's, we will show how  to estimate the nonlinear filter.  Since this procedure is based on the quantization of the Brownian motion and skips the use of algorithms to perform the stationary quantizers, the computation of the marginal quantizers  is quite instantaneous.  This reduce drastically  the time computation of the procedure with respect to the  vector quantization method since we  skip  the step of the  use of algorithms  search  of marginal  quantizers by using instead,   the marginals of the functional  quantization of the signal process.

In the rest of the paper we deal with the estimations methods of the conditional survival probability.

\section{Numerical tools} 
Our aim in this section is to derive  a way to compute numerically the conditional survival probability using Equation  (\ref{EqExplicitCompSurvivalProba}).    To this end, we have to estimate three  quantities: the quantity $\bar{F}(t_m,t_n,\bar X_{t_m})$, the expectation  $\mathds E[\bar{F}(t_m,t_n,\bar X_{t_m}) K^m L_y^m]$ as soon as $\bar{F}(t_{m},t_n, \bar X_{t_m})$ is estimated, and finally, the expectation $\mathds E[L_y^m]$.  Both expectations will be estimated using marginal functional quantization method.  The probability $\bar{F}(t_m,t_n,\bar X_{t_m})$ will be estimated by Monte Carlo simulations. 

%
Before  dealing with  the estimation tools, we shall recall first some basic results about both optimal vector quantization and functional quantization. Then,  we will  show how to construct the marginal functional quantization process which will be used to estimate the quantities of interest. 


\subsection{Overview on optimal quantization methods} 

The optimal vector quantization  on a  grid $\Gamma =\{ x_1,\cdots,x_N \}$  (which will be called a quantizer)   of an $\mathbb R^d$-valued random vector $X$   defined  on a probability space $(\Omega,\mathcal{A},\mathbb{P})$ with finite $r$-th moment and probability distribution $\mathbb P_X$ consists in  finding  the best approximation of  $X$  by a  Borel function of $X$ taking  at most  $N$ values. This turns out to find the solution of the following minimization problem:
\begin{equation}   \label{EqDefErrorQuant}
 e_{N,r}(X)  =  \inf{\{ \Vert X - \widehat{X}^{\Gamma} \Vert_r, \Gamma \subset \mathbb{R}^d,   \vert \Gamma \vert \leq N \}}, 
 \end{equation}
where ${\Vert X \Vert}_r := {\left(\mathbb E[{|X|}^r] \right)}^{1/r}$ and  where $ \widehat{X}^{\Gamma} := \sum_{i=1}^N x_i {1 \! \! 1}_{\{X \in C_i(\Gamma)\}}$  is the quantization of $X$ (we will write $\widehat{X}^N$ instead of $\widehat{X}^{\Gamma}$)  on the grid $\Gamma$ and $(C_i(\Gamma))_{i=1,\dots,N}$  corresponds to a Voronoi tessellation  of  $\mathbb{R}^d$ (with respect to a norm $\vert \cdot \vert$ on $\mathbb{R}^d$), that is, a Borel partition of $\mathbb{R}^d$ satisfying for every $i$, $$ C_i(\Gamma) \subset \{ x \in \mathbb{R}^d : \vert x-x_i \vert = \min_{j=1,\dots,N} \vert x-x_j \vert \}.$$
%
%
%
We know that for every $N \geq 1$, the infimum in $(\ref{EqDefErrorQuant})$ is reached at  one  grid $\Gamma^{\star}$  at least,  called a  $L^r$-optimal $N$-quantizer.  It is also known that if  $\textrm{ card(supp}(\mathbb P_X)) \geq N $ then   $\vert  \Gamma \vert = N$ (see e.g. $\cite{GraLus}$ or $\cite{Pag`es1998}$).  Moreover, the  $L^r$-mean quantization error  $e_{N,r}(X)$ decreases to zero  at  an  $N^{-1/d}$-rate as the size $N$ of  the grid  $\Gamma$ goes to infinity. This convergence rate has been investigated in \cite{BucWis} and  \cite{Zad}  for absolutely continuous probability measures under the  quadratic norm on $\mathbb{R}^d$, and studied in great details  in \cite{GraLus} under an arbitrary norm on  $\mathbb{R}^d$  for  absolutely continuous measures  and some singular measures.

From the numerical integration viewpoint, finding an optimal quantization grid $\Gamma^{\star}$ may be a challenging task. In practice (we will only consider the quadratic case, i.e. when $r=2$)  we are sometimes led to find some ``good'' quantizations ${\rm Proj}_{\Gamma}(X)$ (with $\Gamma = \{x_1,\dots,x_N\}$)  which are close to $X$ in distribution, so that  for every Borel function $F: \mathbb R^d \mapsto \mathbb R$, we can approximate $\mathbb E[ F(X)]$ by  
\begin{equation} \label{QuantProcedureEstim}
\mathbb{E}\left[ F \big(\widehat{X}^N \big)\right] = \sum_{i=1}^N F(x_i)  \ p_i,
\end{equation}
where $p_i =\mathbb{P}\big( {\rm Proj}_{\Gamma}(X) = x_i \big).$  Amount  ``good'' quantizations of $X$  we have stationary quantizers.  A grid $\Gamma$ inducing the quantization $\widehat{X}^N$  of $X$  is said stationary  if  
\begin{equation}
\forall \ i \neq j, \quad \ \ x_{i}\neq x_{j} \ \ \mbox{and} \ \ \mathbb{P}\left(X \in \cup_{i} \partial  C_i(\Gamma) \right)=0
\label{eq11}
\end{equation}
 and
$$
\mathbb{E}\left[X  \vert  \widehat{X}^N \right] =\widehat{X}^N.
$$
The stationary quantizers search is based on zero search recursive procedures like Newton algorithm in the one dimensional framework and  some algorithms as  Lloyd's I algorithms (see e.g. \cite{GerGra}),   the Competitive Learning Vector Quantization (CLVQ) algorithm (see \cite{GerGra}) or stochastic algorithms (see \cite{PagPri03}) in the multidimensional framework.  Note that optimal  quantizers estimates   of the multivariate Gaussian random vector are available in the  website {\tt www.quantize.math-fi.com}.

We next  recall some error bounds induced from approximating $\mathbb E[F(X)]$ by  (\ref{QuantProcedureEstim}). Let $\Gamma$ be a stationary quantizer    and $F$ be a Borel function on $\mathbb R^d$.  
\begin{itemize}
\item[(i)]  If $F$ is convex then
$$
\mathbb{E}\big[F(  \widehat{X}^N)\big]\leq \mathbb{E}[F(X)].
$$

\item[(ii)]Lipschitz functions: 
\begin{itemize}

\item  If $F$ is Lipschitz continuous then (this error bound doesn't require the quantizer $\Gamma$ to be stationary)
$$
\big \vert \mathbb{E}[ F(X) ]- \mathbb{E} [F( \widehat{X}^N) ]  \big \vert \leq [F]_{Lip} \Vert X-   \widehat{X}^N \Vert_{2},
$$
where $$[F]_{{\rm Lip}}:= \sup_{x \not= y} \frac{\vert F(x) - F(y) \vert}{\vert x-y \vert}.$$
\item Let $\theta: \mathbb R^d\rightarrow \mathbb{R}_{+}$ be a nonnegative convex function such that $\theta(X) \in L^{2}(\mathbb{P})$. If $F$ is locally Lipschitz with at most $\theta$-growth, i.e. $\left|F(x)-F(y)\right| \leq [F]_{Lip} |x-y| \left(\theta(x)+\theta(y)\right)$ then $F(X)\in L^{1}(\mathbb{P})$ and
$$
\big \vert \mathbb{E}  [F(X)] - \mathbb{E} [F( \widehat{X}^N)]  \big \vert \leq 2 [F]_{Lip} \Vert X- \widehat{X}^N\Vert_{2} \Vert \theta(X) \Vert_{2}.
$$
\end{itemize}

\item[(iii)]Differentiable functionals: \\
If $F$ is differentiable on $\mathbb R^d$ with an $\alpha$-H\"older differential D$F$ ($\alpha \in [0,1]$), then
$$
\big \vert \mathbb{E}  [F(X)]  - \mathbb{E} [F(\widehat{X}^N )] \big \vert \leq [DF]_{\alpha} \Vert X-  \widehat{X}^N \Vert_{2}^{1+\alpha}.
$$
\end{itemize}
Other error bounds related to the regularity of $F$ may be found in  \cite{PagPri}).

The optimal vector quantization may be extended to random vectors with values in a set of infinite dimension, in particular to stochastic processes viewed as random variables with values in $L^2([0,T],dt)$ endowed with the  norm
$$ \mathbb E \left[\vert X \vert_{L^{2}_{T}}^2\right]  =  \int_{0}^{T} \mathbb E [X^{2}_{s}] ds <+\infty. $$
The functional quantization of the stochastic process $(X_t)_{t \in [0,T]}$ with dynamics 
$$ dX_t = b(X_t,t) dt + \sigma(X_t,t) dW_t , \quad X_0 = x \in \mathbb R, $$
is based on the functional  quantization of the Brownian motion $W$. One way to quantize the Brownian motion is to use the optimal product quantization using its Karhunen-Lo\`eve expansion which reads :
$$
W\; \stackrel{L^{2}_{T}}{=} \;\sum_{n \geq 1} \sqrt{\lambda_{n}} \xi_{n} e_{n},
$$
 where $\xi_{n}:=\left\langle W,e_{n}\right\rangle /\sqrt{\lambda_{n}}$, $n\geq1$, is a sequence of i.i.d. random variables with standard normal distribution and 
$$
e_{n}(t)=\sqrt{\frac{2}{T}} \sin\left(\pi \left(n-\frac{1}{2}\right)\frac{t}{T}\right), \ \ \lambda_{n}=\left(\frac{T}{\pi\left(n-\frac{1}{2}\right)}\right)^{2}, \ \ n \geq 1.
$$
In fact, from the previous expansion, a functional quantization $\widehat W$ of the process $W$ of size at most $N$ is defined by  
\begin{equation}  \label{EqDefQuantW}
 \widehat{W}_{t}^N = \sum_{n\geq1} \sqrt{\lambda_{n}} \hat{\xi}^{x^{N_n}}_{n} e_{n}(t) 
 \end{equation}
where $ \hat{\xi}^{x^{N_{n}}}_{n}$ (with $x^{N_n} = \{x^{N_n}_1,\dots,x^{N_n}_{N_n}  \}$)  is the optimal $N_{n}$-quantization of $\xi_{n}$ and $ N_{1}\times \dots \times N_{n} \leq N$, with $N_i \geq 2$ for $i\leq n$,  and $N_k=1$ for $k \geq n+1$, so that the expansion defined in (\ref{EqDefQuantW}) is a finite sum.    The product quantizer $\chi$ that produces the above Voronoi quantization $\widehat{W}$ is defined by
$$
\chi_{\underline{i}}(t) = \sum_{n\geq1} \sqrt{\lambda_{n}} x_{i_{n}}^{N_n} e_{n}(t), \ \ \ \underline{i}=(i_{1},\cdots,i_{n}, \cdots)\in  \prod_{n\geq1} \left\{1,\cdots,N_{n}\right\}.
$$

\noindent and for every multi-index $\underline{i} \in \prod_{n\geq1} \left\{1,\cdots,N_{n}\right\}$, the associated Voronoi cell of $\chi$ is
$$
C_{\underline{i}}(\chi) = \prod_{n\geq1} \sqrt{\lambda_{n}} C_{i_{n}}(x^{N_n}).
$$
The optimal product  quantizer of size at most $N$, denoted $\widehat{ W}^N$, of the Brownian motion  is defined as the solution  of  the  following optimization problem:
\begin{equation} \label{EquatProbBit}
 \min \big \{ \Vert W- \widehat{W}^N \Vert_{2}, \ N_{1},\cdots,N_{n} \geq 1, \;N_{1}\times \cdots \times N_{n} \leq N,\; N \geq 1  \big\}.
 \end{equation}
 Moreover  this optimal product quantizer induces a rate-optimal sequence of quantizers (see e.g. \cite{LusPag1} for more details), i.e. 
$$
\Vert W-\widehat{W}^N \Vert_{2} \leq K_{W} \frac{T}{(\log N)^{\frac{1}{2}}}
$$
for some real constant $K_{W}>0$. To define a functional quantization of the stochastic process $(X_t)_{t \geq 0}$ consider $(\chi^{N})_{N\geq1}$ a sequence of rate-optimal product quantizers of the Brownian motion and, for every multi-index $\underline{i} \in \prod_{n\geq1} \left\{1,\cdots,N_{n}\right\}$, with $N_{1}\times \cdots \times N_{n} \leq N$, consider $x_{\underline{i}}$ the solution of the following integral equation
\begin{equation} \label{EqQuantX}
dx_{\underline{i}}(t) = \Big(b(x_{\underline{i}}(t),t)-\frac{1}{2} \sigma \sigma'(x_{\underline{i}}(t),t)\Big) dt + \sigma(x_{\underline{i}}(t),t) d\chi_{\underline{i}}(t),
\end{equation}
where $\sigma'$ is the derivative of $\sigma$. We define the functional (non-Voronoi) quantization $\widetilde{X}$ process of the stochastic process $X$ of size at most $N$ by
 \begin{equation*}
\widetilde{X}^N = \sum_{i=1}^{N} x_{\underline{i}} \mbox{\textbf{1}}_{C_{\underline{i}}(\chi)}(W), \quad  N \geq 1.
\end{equation*}
We have the following result. 

\begin{prop}[See \cite{LusPag1}]
Under some suitable conditions on the coefficients of the diffusion, which in the homogeneous case are equivalent to:   $b$ is differentiable, $\sigma$ is positive twice differentiable and  $b'-b\frac{\sigma'}{\sigma}-\frac{1}{2}\sigma\sigma''$ is bounded, we have  
$$
\Vert X-\widetilde{X}^N \Vert_{2} =  \mathcal O\big((\log N)^{-\frac{1}{2}}\big).
$$
\end{prop}
We observe that  for any initial value $x$ of the quantized process $\widehat{X}$, the marginals $\widehat{X}_t$ are of size $d_N =N_1\times \dots \times N_n$ where $(N_1,\dots,N_n)$ is the solution of (\ref{EquatProbBit}).  To define the marginal functional quantization process (still be denoted by $(\widehat X_t)$) of the process  $(X_t)$ we order the values of the marginals and define the marginal quantizer $x^{d_N}(t)$ for every $t>0$ by  $x^{d_N}(t) =\{x_1^{d_N}(t),\dots,x_{d_N}^{d_N}(t) \}$ and $x^{d_N}(0) = x$ and define likewise the marginal quantizers of $W_t$ by   $\chi^{d_N}(t) =\{\chi_1^{d_N}(t),\dots,\chi_{d_N}^{d_N}(t) \}$.    The marginal functional quantization process of size at most $N$ is then defined by 
$$ \widehat X^N_t = \sum_{i=1}^{d_N} x^{d_N}_i (t) \mbox{\textbf{1}}_{\{W_t \in C_i(\chi^{d_N}(t)) \}}. $$
Remark  that the use of the marginal functional quantization method can not be justified from the theoretical point of view since we do not know yet the rate of convergence of the marginals of the quantized process to the marginals of the initial process. However, this method has proved its efficiency from the numerical viewpoint  when used to estimate barrier option by optimal quantization, see \cite{Sag} (not that the considered  marginal quantization in \cite{Sag} is a little bit different from the one considered in this paper, nevertheless the numerical results are the same, up to at least a  $10^{-3}$ absolute error order).   

Let us come back to the problem of interest and let  the functionals $\pi_{y,m}$ and $\varpi_{y,m}$ be defined for every bounded measurable function $f$ by 

$$  \pi_{y,m} f = \mathds E \big[f( \bar X_{t_m}) L_{y}^m K^m \big]  \quad \textrm{and } \  \varpi_{y,m} f =  \mathds E\big[f(\bar X_{t_{m}}) L_{y}^m \big]. $$
Then,  
$$ \Pi_{y,m} \bar F(t_m,t_n,\cdot) := \mathds P\big(  \tau_{\bf a}^{\bar X} > t_n \vert  \,  Y_{t_0}=y_0,\dots, Y_{t_m} =y_m\big) = \frac{\pi_{y,m} \bar F(t_{m},t_n,\cdot)}{\varpi_{y,m} \mbox{\bf 1}},   $$
where $\mbox{\bf 1}(x) =1$, for every real $x$.  Then, it will be enough to show how to estimate $\pi_{y,m} \bar F(t_m,t_n,\cdot)$ since $\varpi_{y,m}  \mbox{\bf 1}$ is estimated similarly.  We discuss in the section below the estimation of  $\pi_{y,m}$.


\subsection{Estimation of $\pi_{y,m}$ by marginal functional quantization}
Our aim is to estimate 
\begin{equation*}
 \pi_{y,m} \bar F(t_{m},t_n,\cdot)  =  \mathds E \Big[ \bar F(t_{m},t_n,\bar X_{t_m}) \prod_{k=0}^{m-1} g_k^{{\bf a}}( \bar X_{t_k},y_k; \bar X_{t_{k+1}}, y_{k+1}) \Big]
 \end{equation*}
by marginal functional quantization  (with $\bar X \equiv X$ in the model  (\ref{DinBS})), where 
 $$ g_k^{{\bf a}}( \bar X_{t_k},y_k; \bar X_{t_{k+1}}, y_{k+1})  =  g_k( \bar X_{t_k},y_k; \bar X_{t_{k+1}}, y_{k+1}) \times G_{\Delta_k \sigma_k^2}^{\bar X_{t_{k}},\bar X_{t_{k+1}}} ({\bf a}) .  $$
  We deal with a general setting: the estimation of   $\pi_{y,m} f$  for any bounded and measurable function $f$.  Our  main reference is  \cite{PagPha} where  $\pi_{y,m} f$ has been estimated  using marginal vector quantization methods.
 \noindent
 Let  us define  for every $k=1,\dots,m$ the transition kernel $H_{y,k}$ by
 \begin{eqnarray*} \label{def_kernel_H}
 H_{y,k} f(x) &  = &  \mathds{E} \left[ f(\bar{X}_{t_k})  g_{k-1}^{{\bf a}}(x,y_{k-1};  \bar X_{t_{k}}, y_{k})  \vert  \bar{X}_{t_{k-1}} =x \right]  \\
 &  = & \int f(z) g_{k-1}^{{\bf a}}(x,y_{k-1};  z, y_{k})     P_k(x,dz) 
 \end{eqnarray*}
where $P_k(x,dz)$ is the density of the random variable $\bar {X}_{t_{k} } \vert \bar X_{t_{k-1}} = x$.  We set 
\begin{equation} \label{def_kernel_H0}
 H_{y,0} f= \mathbb{E}[f(\bar{X}_0)]  = \int f(z) \mu(dz) .
 \end{equation}
 Then, we have   (setting for every $k$,  $\mathcal F_{t_k}^{\bar X} = \sigma (\{ \bar X_{t_i}, i=0,\dots,k \})$)
 \begin{eqnarray*}
\pi_{y,k} f & =  &  \mathds{E} \Big[  \mathbb{E} \Big[ f(\bar{X}_{t_k}) \prod_{i=0}^{k-1}  g_i^{{\bf a}}(\bar X_{t_i},y_i;  \bar X_{t_{i+1}}, y_{i+1})     \big  \vert  \mathcal F_{t_{k-1}}^{\bar X} \Big] \Big]  \\
&  = &  \mathds{E} \Big[  \mathds{E} \Big[ f(\bar{X}_{t_k})  g_{k-1}^{{\bf a}}(\bar X_{t_{k-1}},y_{k-1};  \bar X_{t_{k}}, y_{k})       \vert \mathcal F_{t_{k-1}}^{\bar X} \Big]  \prod_{i=0}^{k-2}   g_i^{{\bf a}}(\bar X_{t_i},y_i;  \bar X_{t_{i+1}}, y_{i+1})  \Big]  \\
&  = &  \mathds{E} \Big[  \mathbb{E} \Big[ f(\bar{X}_{t_k})  g_{k-1}^{{\bf a}}(\bar X_{t_{k-1}},y_{k-1};  \bar X_{t_{k}}, y_{k})     \vert \bar{X}_{t_{k-1}} \Big]  \prod_{i=0}^{k-2} g_i^{{\bf a}}(\bar X_{t_i},y_i;  \bar X_{t_{i+1}}, y_{i+1})    \Big],  
\end{eqnarray*}
where the last equality is a consequence of the Markov property of the process $(\bar{X}_{t_k})$. Thus, we deduce that  for every $k=1,\dots, m$,
\begin{eqnarray*}
\pi_{y,k} f &  = &  \mathbb{E} \Big[  H_{y,k} f(\bar{X}_{t_{k-1}})  \prod_{i=0}^{k-2}   g_i^{{\bf a}}(\bar X_{t_i},y_i;  \bar X_{t_{i+1}}, y_{i+1})    \Big]\\
& = & \pi_{y,k-1} H_{y,k} f.
\end{eqnarray*}
It follows  that  $\pi_{y,m} f $ can be computed by the following recursive formula:
\begin{equation}  \label{EqEstimPiNL}
 \pi_{y,m} f = (H_{y,0}  \circ H_{y,1} \circ \dots \circ H_{y,m}) f.
\end{equation}
Therefore, to achieve the  estimation of  $\pi_{y,m}$, it remains to estimate the kernels $H_{y,k}$.  This will be done by marginal functional quantization.

Consider time discretization  steps $t_{k}$, $k=0,\ldots,m$  and let $\chi^{N_k}:= \{ \chi_k^1,\ldots,\chi_k^{N_k} \}$  be a $N_k$-quantizer of $W_{t_k}$ (we will consider the marginal functional quantization of Brownian motion so that $N_k = d_N$ for every $k$).  Suppose that   we have also  access  to the marginal functional quantization process $(\widehat{X}_{t_k})_{k}$  of the process $(\bar X_t)_{t \geq 0}$  over the time steps $t_k, k=m,\dots,n$:   $\{ x_k^1,\dots,x_k^{N_k} \}$, of sizes $N_k$ (keep in mind that $N_k =d_N$ for every $k$).

 It follows that the transition kernels $H_{y,k}$ may be estimated  for every $k=1,\dots,m$  by 
$$ \widehat{H}_{y,k}  = \sum_{j=1}^{N_k} \widehat{H}_{y,k}^{ij} \delta_{x_{k-1}^i}, $$
where  
\begin{equation}  \label{EqDefHIJNL}
 \widehat{H}_{y,k}^{ij} = g_{k-1}^{{\bf a}}(x^i_{k-1},y_{k-1};  x^j_k, y_{k})  \,  \hat{p}_k^{ij}, \quad i=1,\dots,N_{k-1}; j=1,\dots,N_k
 \end{equation}
and where the $\hat{p}_{k}^{ij}$'s correspond to the estimation of the transition probabilities from   $\widehat{X}_{t_{k-1}} = x_{k-1}^i$ to $\widehat{X}_{t_{k}} = x_k^j$:
\begin{equation} \label{transition_weightsNL} 
\hat{p}_k^{ij} = \mathbb{P} (\widehat{X}_{t_k} = x_k^j  \vert \widehat{X}_{t_{k-1}} = x_{k-1}^i),  \quad  i=1,\dots,N_{k-1}; j=1,\dots,N_{k}.
\end{equation}
 
 Finally, one will  perform the estimation  $\widehat{\pi}_{y,m}= \widehat H_{y,0}  \circ  \widehat H_{y,1} \circ \dots \circ \widehat H_{y,m} $ of $\pi_{y,m}$ as soon as we will be able to compute $\hat{p}_k^{ij}$. In the  proposition below we show how to compute these probabilities from the cumulative distribution function of  the random variable $ W_{t_k} \vert  W_{t_{k-1}} = x$, which has a Gaussian distribution with mean $ x $ and variance $\Delta_k$ (we refer to \cite{Sag} for a similar result).   
%
 \begin{prop}
 The transition probabilities can be estimated  for every $k=1,\dots,m$ by
\begin{equation} \label{EqApproxProbaBrow} 
 \hat{p}_{k}^{ij} \approx  \mathcal N\big(\chi^{j+}_k;\chi^{i}_{k-1}  \big) - \mathcal N\big(\chi^{j-}_{k};\chi^{i}_{k-1} \big),
 \end{equation}
where  $\mathcal N(\cdot\,; \chi^{i}_{k-1} )$ is the cumulative distribution function  of the normal distribution with mean $ \chi^{i}_{k-1}$ and variance $\Delta_k $  and  where  for every $k=1,\dots,m$,
\[
    \left \{ \begin{array}{ll}
 \chi^{j+}_k  = \frac{ \chi^{j+1}_k + \chi^{j}_k }{2}; \;  \chi^{j-}_k  = \frac{ \chi^{j-1}_k + \chi^{j}_k }{2}  \qquad j=1,\dots, N_k \\
 \\
 \chi^{1-}_k  =-\infty; \ \chi^{N_k^{+}}_k =+\infty.
 \end{array}  
 \right.
 \]
 \end{prop}
 
 \begin{proof}[$\textbf{Proof}$]
 One has  for every $k=1,\dots,m$,
\begin{eqnarray*}
 \hat{p}_{k}^{ij}  &  = & \mathds P\big(W_{t_k} \in C_j(\chi^{N_k})  \big\vert  W_{t_{k-1}} \in C_i(\chi^{N_{k-1}} )\big)  \\   
 &  = & \mathds P\big( W_{t_k} \leq \chi^{j+}_k  \big\vert  W_{t_{k-1}} \in  C_i(\chi^{N_{k-1}} ) \big)  - \mathds P\big(W_{t_k} \leq \chi^{j-}_k  \big\vert  W_{t_{k-1}} \in C_i(\chi^{N_{k-1}} )  \big).
\end{eqnarray*}
On the other hand, we have for every $x \in \mathbb R$,
\begin{eqnarray*}
\mathds P\big(W_{t_k} \leq x  \big\vert  W_{t_{k-1}} \in  C_i(\chi^{N_{k-1}} )  \big)  & = &  \frac{\mathds P\big( W_{t_k} \leq x ; W_{t_{k-1}} \in  C_i(\chi^{N_{k-1}} )  \big) }{ \mathds P\big(W_{t_{k-1}} \in C_i(\chi^{N_{k-1}} )  \big)}. 
\end{eqnarray*}  
Set $C_i = C_i(\chi^{N_{k-1}} )$. Then the numerator  on  the right hand side of the previous equation may be expressed as
\setlength\arraycolsep{0.2pt}
\begin{eqnarray}  \label{EqavantApproxquant}
\mathds P\big( W_{t_k} \leq x ; W_{t_{k-1}} \in C_i \big) & = &  \int_{-\infty}^{x}  \Big(\int_{C_i } \mathds P( W_{t_k} \in dx \vert  W_{t_{k-1}} = y) \mathds P({W_{t_{k-1}}} \in dy)  \Big) dx \nonumber \\
& = & \int_{C_i  } \mathcal N(x;y) \mathds P({ W_{t_{k-1}} }\in dy) \\
& \approx & \,  \mathcal N(x; \chi^i_{k-1})   \mathds P\big(W_{t_{k-1}} \in C_i   \big),  \nonumber
\end{eqnarray}
where $\mathcal N(\cdot;y)$ is the cumulative distribution function  of  $W_{t_k}  \vert  W_{t_{k-1}}=y$.   The last quantity  is the  approximation  of (\ref{EqavantApproxquant})  by optimal quantization with one grid point, considering that $\{\chi^i_{k-1} \}$ is  the quantizer of size one of  the random variable $W_{t_{k-1}}$ over the cell $C_i(\chi^{N_{k-1}} )$. 
 \end{proof}

Following the previous approach, the estimations  $\widehat{\pi}_{y,m}$ and  $\widehat{\varpi}_{y,m}$  of   $\pi_{y,m}$ and  $\varpi_{y,m}$ are computed from the following recursive formulae:
\begin{equation}
\left \{
\begin{array}{ll}
\widehat{\pi}_{y,0} = \widehat H_{y,0} \\
\widehat{\pi}_{y,k} = \widehat{\pi}_{y,k-1} \widehat H_{y,k}:=\Big[   \sum_{i=1}^{N_{k-1}} \widehat H_{y,k}^{i,j} \widehat{\pi}_{y,k-1}^{i}  \Big]_{j=1,\dots,N_k},\quad  k=1,\dots,m 
\end{array}
\right.
\end{equation}
where  
$$ \widehat{H}_{y,k}^{ij} = g_{k-1}^{{\bf a}}(x^i_{k-1},y_{k-1};  x^j_k, y_{k})  \,  \hat{p}_k^{ij}, \quad i=1,\dots,d_N; j=1,\dots,d_N; $$
and 
\begin{equation}
\left \{
\begin{array}{ll}
\widehat{\varpi}_{y,0} = \widehat \Upsilon_{y,0} \\
\widehat{\varpi}_{y,k}= \widehat{\varpi}_{y,k-1} \widehat \Upsilon_{y,k}:=\Big[   \sum_{i=1}^{N_{k-1}} \widehat \Upsilon_{y,k}^{i,j} \widehat{\varpi}_{y,k-1}^{i} \Big]_{j=1,\dots,N_k}, k=1,\dots,m 
\end{array}
\right.
\end{equation}
where  
$$ \widehat{\Upsilon}_{y,k}^{ij} = g_{k-1}(x^i_{k-1},y_{k-1};  x^j_k, y_{k})  \,  \hat{p}_k^{ij}, \quad i=1,\dots,N_{k-1}; j=1,\dots,N_k.$$
Then we approximate $\Pi_{y,m}$ by   $\widehat{\Pi}_{y,m}$ given by
$$ \widehat{\Pi}_{y,m}  = \sum_{i=1}^{N_m} \widehat{\Pi}^i_{y,m} \delta_{x_m^i}$$
with
$$  \widehat{\Pi}^i_{y,m}  = \frac{\widehat {\pi}_{y,m}^i}{\sum_{j=1}^{N_m}  \widehat{\varpi}_{y,k-1}^{i}}, \quad i=1,\dots,N_m. $$

\noindent
Recall that  our aim is to estimate $$\Pi_{y,m}  \bar F(t_m,t_n,\cdot) = \frac{\pi_{y,m} \bar F(t_{m},t_n,\cdot)}{\varpi_{y,m} \mbox{\bf 1}},$$ where the function $\bar F(t_m,t_n,\cdot)$ is defined for every $x \geq 0$ by
$$  \bar F(t_m,t_n,x) =   \mathds E \left[ \prod_{k=m}^{n-1} G_{\Delta_k \sigma_k^2}^{\bar{X}_{t_k}, \bar{X}_{t_{k+1}}}({\bf a}) \big\vert\, \bar{X}_{t_m} =x \right].$$
We shall take the estimation:
\begin{equation}  \label{EqPiF(t_n,t_m,)}
 \widehat{\Pi}_{y,m} \bar F(t_m,t_n,\cdot)  = \sum_{i=1}^{N_m}\widehat{\Pi}^i_{y,m} \bar  F(t_m,t_n,x_m^i).  
 \end{equation}
Remark that  the initial  function   $\displaystyle F(s,t,X_s) = \mathbb{P} \left(\inf_{s \leq u \leq t} X_u >a \vert X_s \right)$   (which has been  estimated by $\bar F(s,t,\bar X_s)$) has semi-closed expression  in some specific models, like  in the model  (\ref{DinBS}) in which case it is  given by 
\begin{equation}  \label{PsurvieW}
\mathbb{P} \left(\inf_{s \leq u \leq t} X_u > {\bf a} \vert X_s \right)  =  N (h_1(X_s,t-s)) - \left( \frac{{\bf a}}{X_s}\right)^{\sigma^{-2}(\mu-\sigma^2/2)} N (h_2(X_s,t-s))
\end{equation}
where
\begin{eqnarray*}
h_1(x,u) &=& \frac{1}{\sigma \sqrt{u}} \Big( \log (x/ {\bf a} )  + \big(\mu - \frac{1}{2} \sigma^2\big) u \Big), \\ h_2(x,u) &=& \frac{1}{\sigma \sqrt{u}} \Big( \log ({\bf a} / x)  + \big(\mu - \frac{1}{2} \sigma^2\big) u \Big)
\end{eqnarray*}
and where $N(\cdot)$ is the cumulative distribution function of the standard Gaussian distribution. 

 Except in these specific cases, the function $\bar F(t_m,t_n,\cdot)$  has  to be estimated, and we shall do it by Monte Carlo methods.

\subsection{Estimation of $\bar F(t_m,t_n,\cdot)$ by Monte Carlo}
It remains to estimate the function $\bar F(t_m,t_n,x)$ defined by
$$  \bar F(t_m,t_n,x) =   \mathds E \left[ \prod_{k=m}^{n-1} G_{\Delta_k \sigma_k^2}^{\bar{X}_{t_k}, \bar{X}_{t_{k+1}}}({\bf a}) \big\vert\, \bar{X}_{t_m} =x \right] $$
by Monte Carlo simulations.  The steps of the Monte Carlo procedure are the following.
    \begin{enumerate}
    \item Let us consider regular time discretization steps  $\{t_{m},t_{m+1},  \dots,t_n\}$ over $[t_m,t_n]$ and let $M$ be the number of trials. We simulate for every $j=1,\dots,M$ the  sample path  $(\bar{X}_{t_{k}}^j)_{k=m,\dots,n}$, with $\bar X^j_{t_m}=x$ for every $j$.
    \item Setting 
    $$p_{m,n}^j(x,{\bf a}):= \prod_{k=m}^{n-1} G_{\Delta_k \sigma_k^2}^{\bar{X}^j_{t_{k}}, \bar{X}^j_{t_{k+1}} }({\bf a} ),$$
    we estimate  $\bar F(t_m,t_n,x)$ by 
    \begin{equation} \label{DefFuncEstiMCFPhi}
     \bar F^M(t_m,t_n,x) =  \frac{1}{M} \sum_{j=1}^M p^j_{m,n}(x,{\bf a}).
     \end{equation}
    \end{enumerate}
    Consequently, integrating  both formulae (\ref{EqPiF(t_n,t_m,)}) and (\ref{DefFuncEstiMCFPhi}),   the conditional survival probability $$\Pi_{y,m} \bar F(t_m,t_n,\cdot)$$ will be estimated (for a fixed  trajectory $(y_0,\dots,y_m)$ of the observation process $(Y_{t_0},\dots,Y_{t_m})$) by
    \begin{equation} \label{EqEstimationFinalForm}
  \widehat{\Pi}_{y,m}  \bar F^M(t_m,t_n,\cdot)  = \frac{1}{M} \sum_{i=1}^{N_m}  \sum_{j=1}^{M}   \widehat{\Pi}^i_{y,m}  p^j_{m,n}(x_m^i,{\bf a}).  
\end{equation} 

\begin{rem}  \label{RemMCerror}
It follows from Monte Carlo error analysis  that for every $x \geq 0$,
\begin{equation}  \label{EqMCError}
\big\Vert  \bar F(t_m,t_n,x)  -   \bar F^M(t_m,t_n,x)  \big\Vert_2   = \mathcal O \Big(\frac{1}{\sqrt{M}} \Big).
\end{equation}
\end{rem}


 \subsection{Error analysis}
 We target to  give in this section  the  error bound resulting from  the estimation of 
$$  \mathds P(\tau_{{\bf a}}^{X} >t \vert Y_{t_0},\dots, Y_{t_m}) $$
 by 
 $$  \widehat{\Pi}_{y,m}  \bar{F}^M(t_m,t_n,\cdot)  = \frac{1}{M} \sum_{i=1}^{N_m}  \sum_{j=1}^{M}   \widehat{\Pi}^i_{y,m}  p^j_{m,n}(x_m^i,{\bf a})  .$$
 We shall first give the error bound due to the fact that we compute the survival function of the random variable $\tau_{\bf a}^{\bar X}$ instead of $\tau_{\bf a}^{X}$, and then the error bound due to the simulations procedures.
 To this end,  the following definitions and assumptions are needed.

\begin{defn}  A probability transition $P$ on $E$ is ${\rm C}$-Lipschitz  (with ${\rm C}>0$) if for any Lipschitz function $f$ on $E$ with ratio $[f]_{Lip}$, $Pf$ is Lipschitz with ratio $[Pf]_{Lip} \leq {\rm C} [f]_{Lip}$.  Then, one may define the Lipschitz ratio $[P]_{Lip}$ by
\begin{equation*}
[P]_{Lip} = \sup \Big\{ \frac{[Pf]_{Lip}}{[f]_{Lip}}, f  \textrm{ a  nonzero Lipschitz function }  \Big\} < +\infty .
\end{equation*}
\end{defn}

Remark that that in our framework,   the transition  operators  $P_k(x,dy),\ k=1,\cdots,m$   are Lipschitz,  so  that we set  
$$ [P]_{Lip} := \max_{k=1,\cdots,m}[P_k]_{Lip} <+\infty. $$
We furthermore define some useful quantities which appear in the error bound in the following. 

    \begin{itemize}
    \item [(i)] For every $k=1,\cdots,m$, we set
        $$ {\rm K}^m_g := \max_{k=1,\cdots,m} \Vert  g_{k}  \Vert_{\infty},$$
        where $\Vert  g_{k}  \Vert_{\infty}$ is the supremum norm of the functions $g_k$ defined by (\ref{EqDefFunctg_k}).
    \item[(ii)]  For every $k=1,\cdots,m$, let  $[g^1_{k}]_{Lip}$ and $[g^2_{k}]_{Lip}$ be so that for every $x,x',\widehat{x},\widehat{x}' \in \mathbb R$ and $y,y' \in \mathbb{R}$,
$$  \vert g_{k}^{\bf a}(x,y,x',y') - g_{k}^{\bf a}(\widehat{x},y,\widehat{x}',y') \vert  \leq [g^1_{k}]_{Lip}(y,y')\ \vert x-\widehat{x} \vert + [g^2_{k}]_{Lip}(y,y') \ \vert x'-\widehat{x}' \vert .$$
    \end{itemize}

Let us make now some assumptions which will be used   to compute (see \cite{GobThese})  the convergence rate  of the quantity  $\mathds E \big\vert {1 \! \! 1}_{\{ \tau^{\bar X}>t \}} - {1 \! \! 1}_{\{ \tau>t \}} \big\vert$ towards   $0$. 
\begin{itemize}
\item [(\textbf {H1})] $b$ is a $\mathcal C_b^{\infty}(\mathbb R,\mathbb R)$ function and $\sigma$ is in $\mathcal C_b^{\infty}(\mathbb R,\mathbb R)$.
\item [(\textbf {H2})] there exists $\sigma_0>0$ such that $\forall x \in \mathbb R, \sigma(x)^2 \ge \sigma_0^2  $ (\emph {uniform ellipticity}).

\end{itemize}

Before giving the error bound associated to our estimation we recall  the following useful results.  Consider in this scope  that
$$  \tau^{X} = \inf\{u \geq0, X_u \not\in D  \}$$
where $D=({\bf a},+\infty)$ and $X$ is the signal process. Let $(\bar{X}_{t_k})_{k=0,\dots,m}$ the continuous Euler process taken at discrete times $t_k,  k=0,\dots,m$ and 
$$ \tau^{\bar X}= \inf\{u \geq0, \bar{X}_u \not\in D  \}.$$
We have the following result.

\begin{prop}[see \cite{GobThese}]  \label{PropGob}
Let   $t>0$.  Suppose that Assumptions   (\textbf{H1}) and (\textbf{H2}) are fulfilled. Then,   for every $\eta \in (0,\frac{1}{2}[$   there exists an increasing function $K(T)$ such that for every $t \in [0,T]$ and for every $x \in \mathbb R$,
$$
\mathbb E_x \left[\big\vert  {1 \! \! 1}_{\{\tau^X > t \}}  - {1 \! \! 1}_{\{\tau^{\bar X}> t\}}  \big\vert \right] \le \frac{1}{ n^{\frac{1}{2}-\eta}} \frac{K(T)}{\sqrt{t}},
$$
 where  $n$ is the number of time discretization steps over $[0,t]$. 
\end{prop}

 The convergence rate of the filter approximation  is given by the following theorem.  Since  we do not know the convergence rate and some properties as the stationary property  of  the marginals of the functional quantization we  consider  here that  for every $k$,   $\widehat{X}_{t_k}^{N_k}$ denotes  the marginal quantization  of  $(X_{t_k})$ of size $N_k$, obtained  from the marginal vector quantization method (see \cite{PagPha}).  

\begin{thm} \label{ThmConvergence}  (see \cite{PagPha} for a similar result).  We have for every $p \ge 1$,
\begin{eqnarray*}
\vert   \Pi_{y,m} \bar F(s,t,\cdot) - \widehat{\Pi}_{y,m} \bar F (s,t,\cdot)   \vert  &  \leq & \frac{{\rm K}^m_g }{\phi_m(y) \vee  \hat{\phi}_m(y)}  \sum_{k=0}^m {\rm B}_k^m(\bar F(s,t,\cdot),y,p) \ \Vert \bar{X}_{t_k} - \widehat{X}^{N_k}_{t_k}  \Vert_p  \\
\end{eqnarray*}
where 
$$ \phi_{m}(y) := \pi_{y,m} \mbox{\bf 1},  \quad \widehat{\phi}_{m}(y) := \widehat{\pi}_{y,m}  \mbox{\bf 1} $$  and  where
\begin{eqnarray*}
{\rm B}_k^m(f,y,p) & := &  (2-\delta_{2,p}) [P]_{Lip}^{m-k} [f]_{Lip} + 2 \bigg( \frac{\Vert f\Vert_{\infty}}{K_g^m} \big([g_{k+1}^1]_{Lip}(y_k,y_{k+1}) +[g_{k}^2]_{Lip}(y_{k-1},y_{k}) \big) \\
& + &  (2-\delta_{2,p})  \frac{\Vert f \Vert_{\infty}}{K_g^m} \sum_{j=k+1}^m  [P]_{Lip}^{j-k-1} \big([g_{j}^1]_{Lip}(y_{j-1},y_{j}) +[P]_{Lip} [g_{j}^2]_{Lip}(y_{j-1},y_{j}) \big) \bigg)
\end{eqnarray*}
(with the convention that  $g_0 = g_{m+1} \equiv 0$ and $\delta_{n,p}$ is the usual Kronecker symbol).
\end{thm}

\begin{proof}[$\textbf{Proof}$]  The proof is similar to the proof of  Theorem $3.1$ in \cite{PagPha}.

\end{proof}
\begin{rem}  One may  remark that  the function  ${\rm B}_k^m(f,y,p)$ involves the norm $\Vert  f \Vert_{\infty}$ which for $f = \bar  F(s,t,\cdot)$ is bounded by $1$.
\end{rem}

Let  us give now the error  bounds induced from the approximation of $\mathds P(\tau_{\bf a}^X>t_n \vert  Y_{t_0},\dots,Y_{t_m})$ by   $ \widehat{\Pi}_{y,m}  \bar{F}^M(t_m,t_n,\cdot) $.

\begin{thm} \label{ThmErrorBounds} (See \cite{CalSag} for a similar result)
Suppose   that the coefficients $b$ and $\sigma$ of the continuous signal process $X$ are such that Assumptions    \textbf{(H1)} and  \textbf{(H2)} are satisfied and let $\eta \in (0,\frac{1}{2}]$. Then,  
\begin{eqnarray*}
\Big \vert \mathds P(\tau_{\bf a}^X>t_n \vert  Y_{t_0},\dots,Y_{t_m}) &  - &      \frac{1}{M} \sum_{i=1}^{N_m}  \sum_{j=1}^{M}   \widehat{\Pi}^i_{y,m}  p^j_{m,n}(x_m^i,{\bf a})   \Big \vert     \leq   \mathcal O \Big(n^{-\frac{1}{2}+\eta} \Big)  + \mathcal O  \Big( M^{-\frac{1}{2}} \Big)  \\ 
& + &  \frac{{\rm K}^m_g }{\phi_m(y) \vee  \hat{\phi}_m(y)}  \sum_{k=0}^m {\rm B}_k^m(\bar F(s,t,\cdot),y,p) \ \Vert \bar{X}_{t_k} - \widehat{X}^{x^{N_k}}_{t_k}  \Vert_p,
\end{eqnarray*}
where the functions $\phi_{m}, \widehat{\phi}_{m}, {\rm B}_k^m(\cdot,y,p)$ are introduced in Theorem \ref{ThmConvergence}.
\end{thm} 

\begin{proof}[$\textbf{Proof}$]  We have 
\begin{eqnarray} \label{EqDecomposError}
\Big \vert \mathds P(\tau_{\bf a}^X>t_n \vert  Y_{t_0},\dots,Y_{t_m}) &  - &      \frac{1}{M} \sum_{i=1}^{N_m}  \sum_{j=1}^{M}   \widehat{\Pi}^i_{y,m}  p^j_{m,n}(x_m^i,{\bf a})   \Big \vert   \nonumber  \\ 
&  \leq & \big \vert \mathds P(\tau_{\bf a}^X>t_n \vert  Y_{t_0},\dots,Y_{t_m}) -  \mathds P(\tau_{\bf a}^{\bar X}>t_n \vert  Y_{t_0},\dots,Y_{t_m}) \big \vert  \nonumber  \\
& &  + \ \big \vert   \Pi_{y,m} \bar F(t_m,t_n,\cdot) - \widehat{\Pi}_{y,m} \bar F (t_m,t_n,\cdot) \big \vert   \nonumber \\
&  & +\  \big\vert  \widehat{\Pi}_{y,m} \bar F (t_m,t_n,\cdot)  -   \widehat{\Pi}_{y,m}  \bar F^M(t_m,t_n,\cdot)  \big \vert. 
\end{eqnarray}
Owing to Remark \ref{RemMCerror} and to the fact that $ \sum_{i=1}^{N_m}  \widehat{\Pi}^i_{y,m} \leq 1$ we get  
\begin{eqnarray*}
\big \vert  \widehat{\Pi}_{y,m} \bar F (s,t,\cdot)  -    \widehat{\Pi}_{y,m}  \bar F^M(t_m,t_n,\cdot)    \big  \vert  & = &    \big\vert \sum_{i=1}^{N_m}  \widehat{\Pi}^i_{y,m} \bar F (t_m,t_n, x_m^i) -  \sum_{i=1}^{N_m}  \widehat{\Pi}^i_{y,m}  \bar F^M (t_m,t_n, x_m^i)  \big\vert \\
& \leq &  \sum_{i=1}^{N_m}  \widehat{\Pi}^i_{y,m} \sup_{x \geq 0} \vert  \bar F (t_m,t_n, x) - \bar F^M (t_m,t_n, x)\vert \\
& \leq & \sup_{x \geq 0} \vert   \bar F (t_m,t_n, x) - \bar F^M (t_m,t_n, x)\vert \\
& = &  \mathcal O\big( \frac{1}{\sqrt{M}}\big).
\end{eqnarray*}
On the other hand, the error  bound of the term  $\big \vert   \Pi_{y,m} \bar F(t_m,t_n,\cdot) - \widehat{\Pi}_{y,m} \bar F (t_m,t_n,\cdot) \big \vert $   is given by Theorem \ref{ThmConvergence}.  \\
\noindent
Now, let us consider the first term of the right hand side of Equation (\ref{EqDecomposError}). We have:
\begin{eqnarray*}
\mathds E\Big[\big \vert \mathds P(\tau_{\bf a}^X>t_n \vert  Y_{t_0},\dots,Y_{t_m})  & - &  \mathds P(\tau_{\bf a}^{\bar X}>t_n \vert  Y_{t_0},\dots,Y_{t_m}) \big \vert \Big]\nonumber \\
 &  = & \mathds E \left[ \big\vert    \mathds E\left[ {1 \! \! 1}_{\{\tau_{\bf a}^X >t_n\}} - {1 \! \! 1}_{\{\tau_{\bf a}^{\bar X}>t_n \}} \vert Y_{t_0},\ldots,Y_{t_m}   \right]  \big\vert \right] \nonumber \\
 & \leq & \mathds E \left[\big\vert    {1 \! \! 1}_{\{\tau_{\bf a}^X >t_n\}} - {1 \! \! 1}_{\{\tau_{\bf a}^{\bar X}>t_n \}}  \big\vert  \right]\nonumber \\
&  = &  \mathcal O(n^{-\frac{1}{2} +\eta}),
\end{eqnarray*}
the last statement following from Proposition \ref{PropGob}.

\end{proof}

 \subsection{Numerical examples}
We deal with numerical simulations in this section by considering  two example of models. \\
In both examples  we fix $t_m=1$ and, given a (simulated) trajectory of the observation process $Y$ from $0$ to $t_m$, we estimate the conditional cumulative function $\mathds P(\tau_{\bf a}^{\bar X}>t_n \vert Y_{t_0},\ldots,Y_{t_m})$  using  formula (\ref{EqEstimationFinalForm}), for  $t_n$ varying $0.1$ by $0.1$ from $1.1$ to $11$ (where the time unit is expressed  in years). Furthermore, we set the number $m$ of discretization points over $[0,t_m]$ equal to $50$ and for every $k =1, \cdots,m$,  the quantization grid size $N_k$ is set to $966$ (as a consequence of the numerical solution of the Problem \ref{EquatProbBit} for $N=10000$, with the optimal decomposition $(23,7,3,2)$, see \cite{PagPri} for more detail), with $N_0 =1$. All the programs have been coded using the {\bf C} language  on a CPU  $2.7$  GHz  and   4 Go memory computer.

\begin{example}[The ``Black-Scholes'' example] \ \\
 The  first model is the one considered in Example \ref{exa:BS} and Corollary \ref{CorExampleDinBS}  where the dynamics of the signal  process $X$ and the observation process $Y$ are given by
 \begin{equation} \label{ModelBlackScholes}
\begin{cases}
dX_t  =  X_t ( \mu dt  +  \sigma  dW_t), &  X_0 = x_0, \\
d Y_t  =  Y_t (\mu dt + \sigma d W_t + \delta d \widetilde  W_t), & Y_0=y_0
\end{cases}
\end{equation}
or  equivalently $$ \frac{dY_t}{Y_t} = \frac{dX_t}{X_t} + \delta d\widetilde{W}_t.$$

 We choose the following parameters (as in \cite{CocGemJea}) of the model: $\mu = 0.03$, $\sigma=0.03$, $x_0=y_0=86.3$ and ${\bf a} = 76$. The numerical results are depicted  in Figure \ref{figure1} and  Figure \ref{figure2}.   In Figure \ref{figure1}, we draw three trajectories of the observation process $Y$ for the same $\delta =0.1$ (left side graphic) and the corresponding cumulative functions $\mathds P(\tau^{\bar X}_{\bf a} \leq t_n \vert Y_{t_0},\ldots,Y_{t_m})$ on the right hand side graphics, $t_m=1$ and $t_n \in [1.1,11]$, in years.  We remark that, for a fixed time $t_n$,  the lower the trajectory is, the higher its probability is to hit the barrier.  

 The  left hand side graphics  of Figure \ref{figure2} corresponds to three trajectories of the observation process $Y$ for $\delta \in \{0.1,0.3,0.5  \}$ and the right hand side graphics, to (a zoom of) the corresponding  cumulative functions $\mathds P(\tau^{\bar X}_{\bf a} \leq t_n \vert Y_{t_0},\ldots,Y_{t_m})$, with $t_m=1$ and $t_n \in [1.1,11]$, in years. We observe in this example that the noisier the observations are, the higher the probability is to hit the barrier ${\bf a}$ before a fixed time $t_n$.  
 
 Note that in both examples, the function $ F(t_m,t_n,\cdot)$ has been computed using formula (\ref{PsurvieW}) and the computation time to get one cumulative function $\mathds P(\tau^{\bar X}_{\bf a} \leq t_n \vert Y_{t_0},\ldots,Y_{t_m})$ for a given $t_n$ is about of $24 $ seconds.

\begin{figure}[htpb]
\begin{center}
   \includegraphics[width=7.7cm,height=6.0cm]{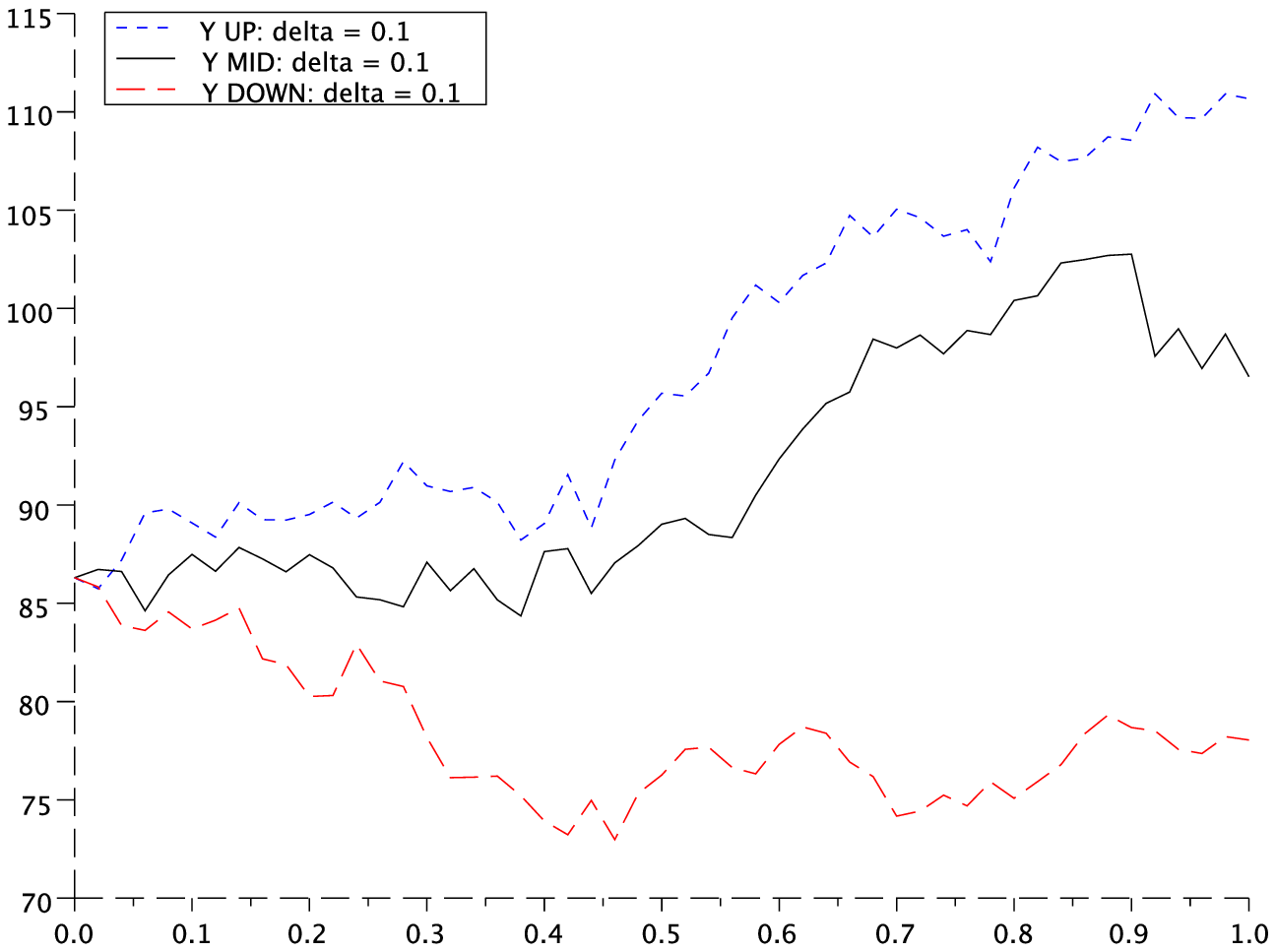}	
 \hfill    \includegraphics[width=7.7cm,height=6.0cm]{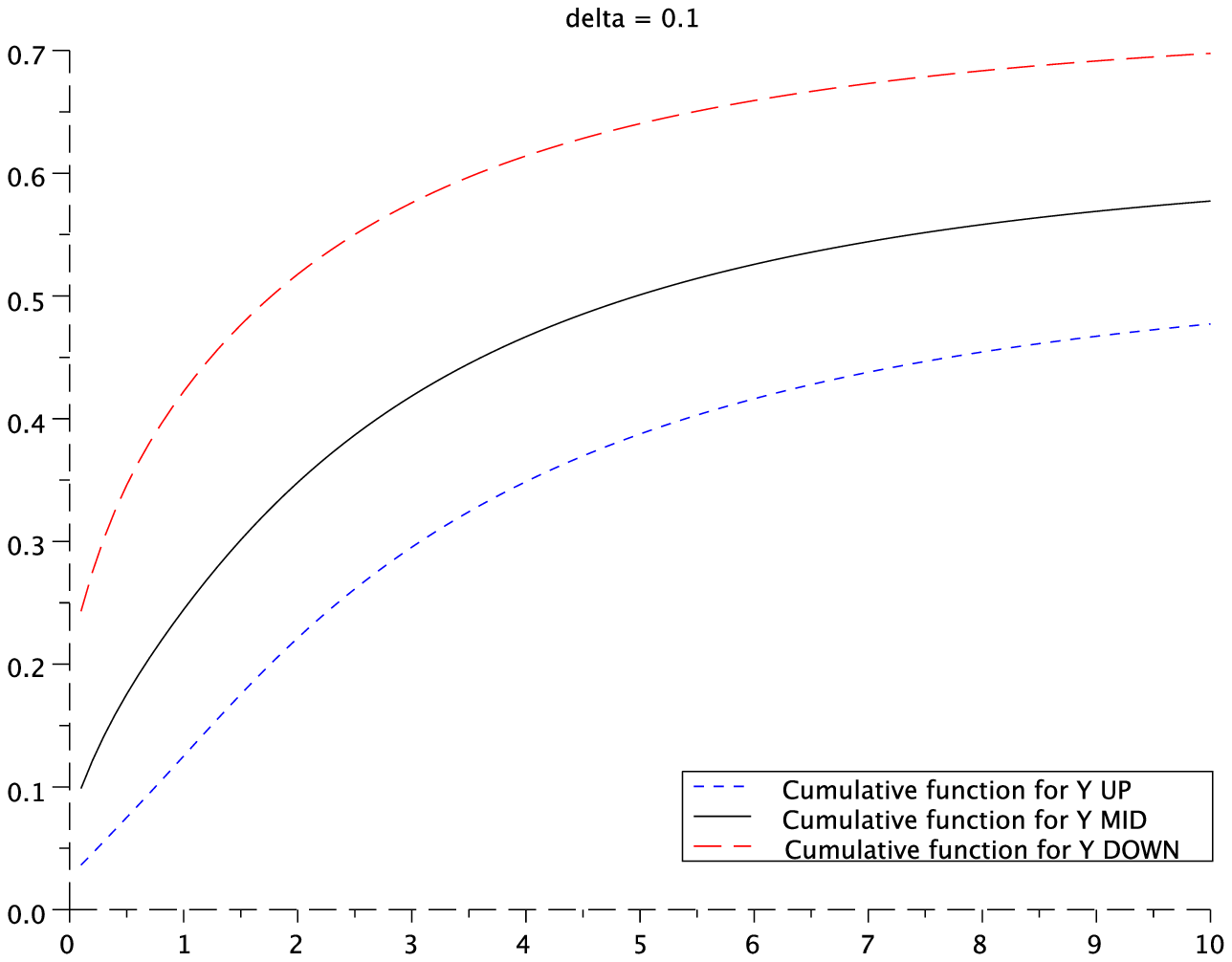}
  \caption{\footnotesize  ("Black Scholes example") Three trajectories for the observation process $Y$ and for  $\delta = 0.1$ (on the left), and the corresponding cumulative functions $\mathds P(\tau^{\bar X}_{\bf a} \leq t_n \vert Y_{t_0},\ldots,Y_{t_m})$ with $t_m=1$ and $t_n \in [1.1,11]$ years (on the right).}
\label{figure1}
  \end{center}
\end{figure}

 \begin{figure}[htpb]
   \includegraphics[width=7.7cm,height=6.0cm]{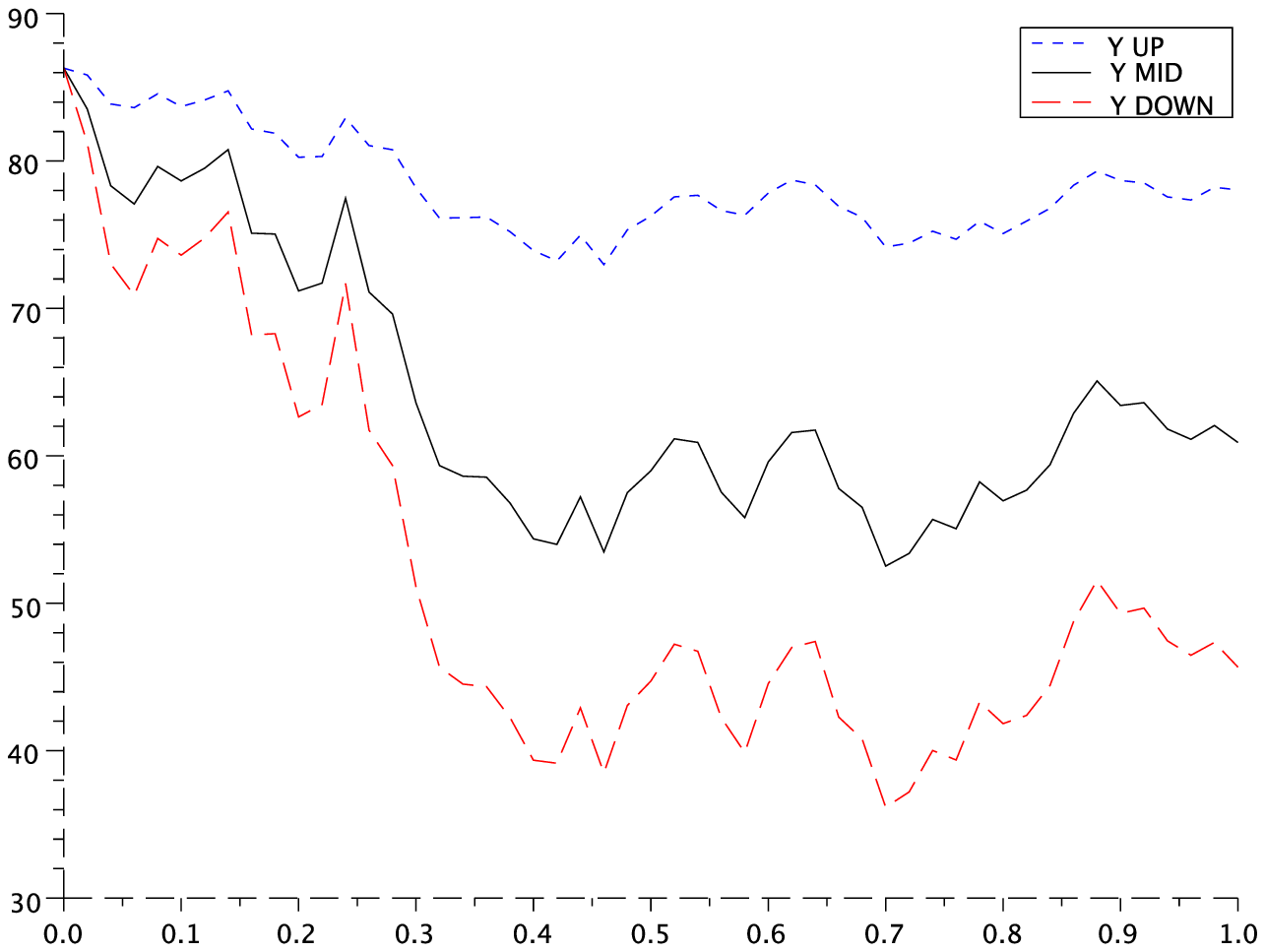}	
  \hfill    \includegraphics[width=7.7cm,height=6.0cm]{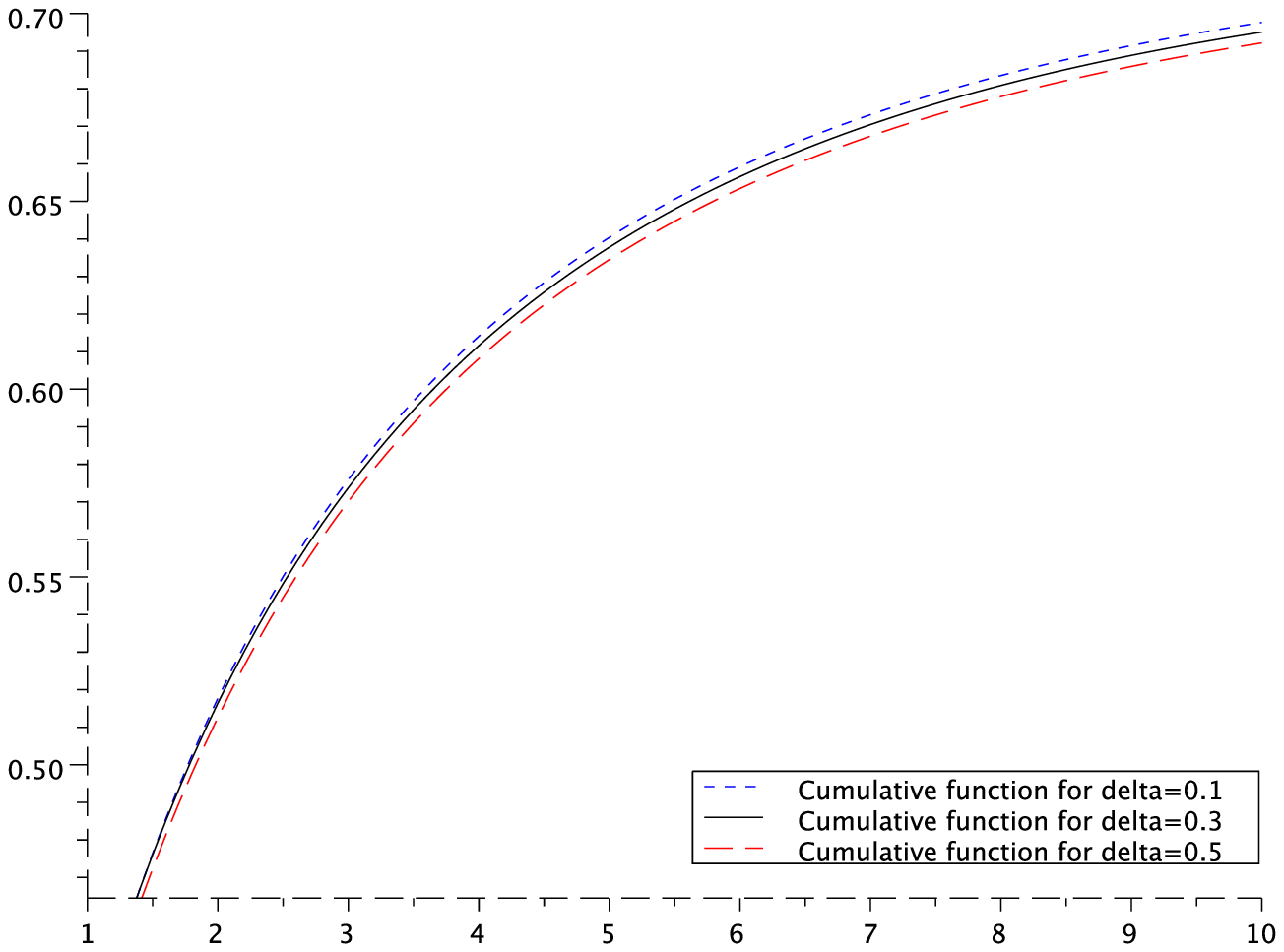}
  \caption{\footnotesize ("Black Scholes example")  Left graphics  corresponds to three trajectories for the observation process $Y$, for  $\delta= 0.1$ ({\tt  Y UP}), $\delta = 0.3$ ({\tt  Y MID}), $\delta = 0.5$ ({\tt  Y DOWN}), and the right hand side graphics  correspond to a zoom of the cumulative functions $\mathds P(\tau^{\bar X}_{\bf a } \leq t_n \vert Y_{t_0},\ldots,Y_{t_m})$ with $t_m=1$ and $t_n \in [1.1,11]$ years.} 
\label{figure2}
\end{figure}

 \end{example}

\begin{example}[The Ornstein-Uhlenbeck example]\ \\
In the second model, we suppose that both the signal and the observation process evolve following the Ornstein-Uhlenbeck dynamics:
\begin{equation}
\begin{cases}
dX_t  =  \lambda (\theta- X_t)  dt  +  \sigma  dW_t, &  X_0 = x_0, \\
d Y_t  =  \lambda ( \theta - Y_t)  dt + \sigma d W_t + \delta d \widetilde  W_t, & Y_0=y_0
\end{cases}
\end{equation}
or  setting $Z_t = Y_t - X_t$
$$ dZ_t = - \lambda Z_t dt  + \delta d\widetilde W_t, $$
meaning that $Z$ is still an  Ornstein-Uhlenbeck process with mean value $\theta=0$ and with volatility $\delta.$ The parameters are chosen as follows: $\lambda = 0.18$, $\theta = 0.35$, $\sigma =0.12$, $x_0=y_0=0.35$ (as in \cite{CocGemJea}) and ${\bf a} =0.2$.   The numerical results are represented in Figure \ref{figure3} where we depict  three trajectories of the observation process $Y$ for $\delta = 0.16$ (left hand side graphics of Figure \ref{figure3}) and  the associated  cumulative functions $\mathds P(\tau^{\bar X}_{\bf a} \leq t_n \vert Y_{t_0},\ldots,Y_{t_m})$, with $t_m=1$ and $t_n \in [1.1,6]$ in years (right hand side graphics of Figure \ref{figure3}).  Once again, we remark that, as in the "Black Scholes example",  for a fixed time $t_n$,   the lower the trajectory is, the higher its probability is to hit the barrier.

In this example,  the function $\bar F(t_m,t_n,\cdot)$ has been computed using Monte Carlo simulations  of size $M=10^5$(see the formula (\ref{DefFuncEstiMCFPhi})) with $50$ discretization steps over $[t_m,t_n]$. The computation time to get one cumulative function $\mathds P(\tau^{\bar X}_{\bf a} \leq t_n \vert Y_{t_0},\ldots,Y_{t_m})$ for a given $t_n$ is about  $6.5$ minutes.

 \begin{figure}[htpb]
   \includegraphics[width=7.7cm,height=6.0cm]{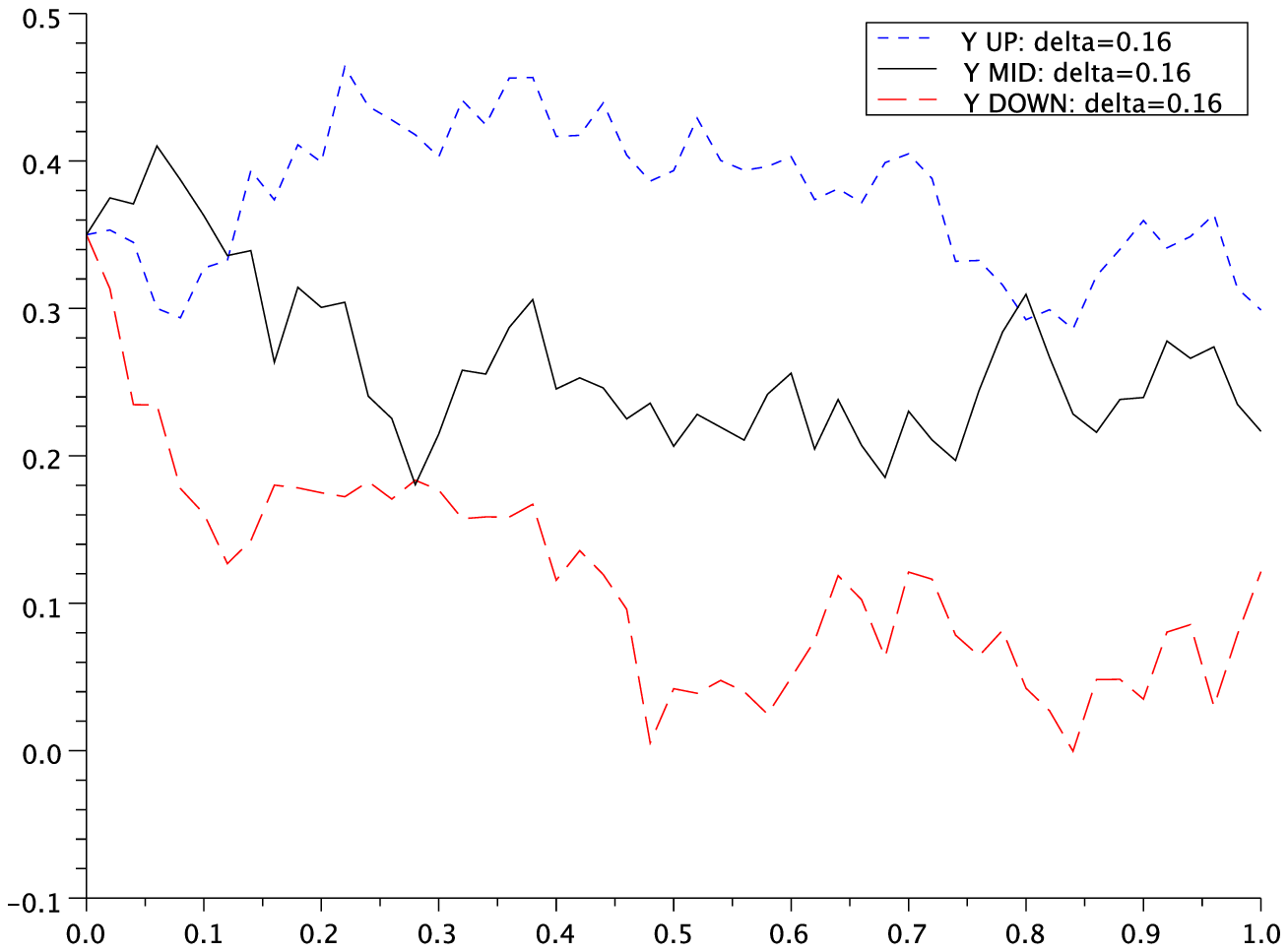}	
  \hfill    \includegraphics[width=7.7cm,height=6.0cm]{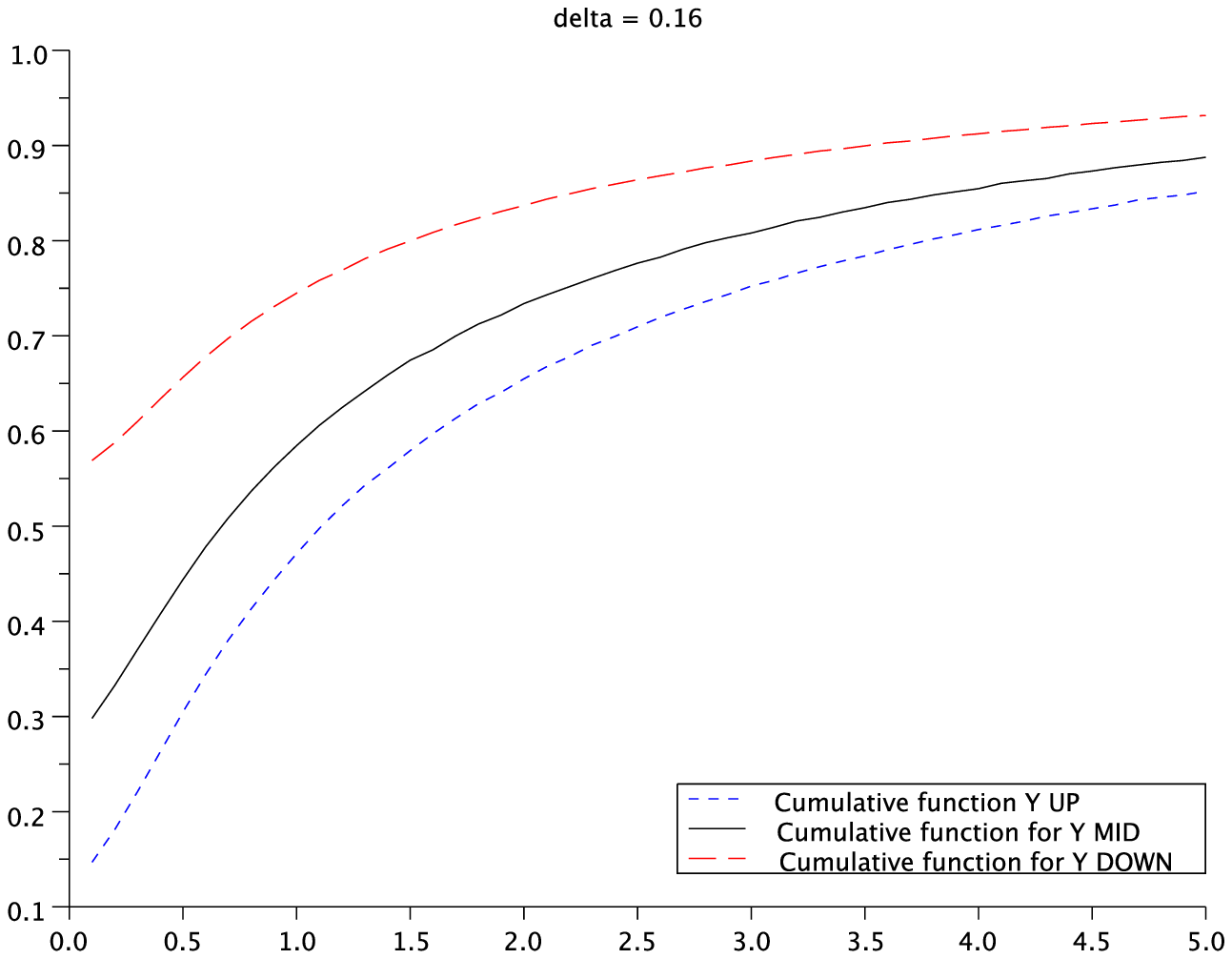}
  \caption{\footnotesize  ("Ornstein-Uhlenbeck example") Left graphics  correspond to three trajectories for the observation process $Y$, for  $\delta= 0.16$   and the right hand side graphics  correspond to  the cumulative functions $\mathds P(\tau^{\bar X}_{\bf a} \leq t_n \vert Y_{t_0},\ldots,Y_{t_m})$ with $t_m=1$ and $t_n \in [1.1,6]$ years.} 
\label{figure3}
\end{figure}

 \end{example}

\newpage

\end{document}